\documentclass[12pt,a4paper,twoside, reqno]{amsart}
\usepackage{amsfonts, amsthm, amsmath, amssymb}
\usepackage{mathrsfs,amsmath}
\usepackage{hyperref}
\hypersetup{colorlinks=false}

\usepackage[margin=2.6cm]{geometry}

\usepackage{helvet}

\usepackage{amsthm}
\newtheorem{theorem}{Theorem}
\newtheorem{lemma}{Lemma}
\newtheorem{remark}{Remark}

\begin{document}
	
	\author{Mohd Harun, Sumit Kumar, Saurabh Kumar Singh}
	
	\title{ Hybrid subconvexity bound for $GL(3)\times GL(2)$ $L$-functions: $t$ and level aspect}
	\address{ Mohd Harun \newline {\em Department of Mathematics and Statistics, Indian Institute of Technology, Kanpur, India; \newline  Email: harunmalikjmi@gmail.com
	} }	
	
	\address{Sumit Kumar \newline {\em Alfr\'ed R\'enyi Institute of Mathematics, Budapest, Hungary;\newline 
			Email: sumit@renyi.hu
	}}
	
	\address{ Saurabh Kumar Singh \newline {\em Department of Mathematics and Statistics, Indian Institute of Technology, Kanpur, India; \newline  Email: skumar.bhu12@gmail.com
	} }
	
	\date{\today}

	\maketitle
	\begin{abstract}
		In this article, we will get non-trivial estimates for the central values of degree six Rankin-Selberg $L$-functions $L(1/2+it, \pi \times f)$  associated with a ${GL(3)}$ form  $\pi$ and a ${GL(2)} $ form $f$ using the delta symbol approach in the hybrid ${GL(2)}$ level and $t$-aspect. 
	\end{abstract}

	\section{Introduction}
	In this article, we yet again explore the subconvexity problem (ScP) for the degree six Rankin-Selberg $GL(3) \times GL(2)$ $L$-functions using Munshi's delta method \cite{ICM}. This theme was initiated by Munshi \cite{r1}   and was explored further by the second author and the third author along with Sharma, Mallesham, and various others (see \cite{SK}, \cite{r21}, \cite{r22}, \cite{r23}, \cite{PS}) to resolve subconvexity for these $L$-functions in various aspects ($t$, twist and spectral aspect). Apparently, the delta method approach so far has been effective for $L$-functions associated with forms admitting a  varying $GL(1)$ factor, barring a few results (see \cite{SK}, \cite{r21}, \cite{r23}), where spectral parameters of a  higher degree form vary, or levels of two higher degree forms vary simultaneously.  In the latter case, the authors produce subconvexity in some relative range of the levels and their method does not allow them to fix one of the forms. As a matter of fact the level aspect subconvexity for automorphic $L$-functions of degree at least three (except for a few instances of  Rankin-Selberg $L$-functions of degree $4$ and $8$)  still remains an important open problem. However,  this problem is more amenable if one allows other parameters to vary along with the level of a form.  The aim of this article is to prove such a result for $GL(3) \times GL(2)$ Rankin-Selberg $L$-functions when archimedean parameters and level of the $GL(2)$ form vary in some relative range. 
	\begin{theorem}\label{main theorem}
		Let $\pi$ be a Hecke-Maass cusp form for $SL(3,\mathbb{Z})$  and $f$ be a holomorphic Hecke cusp form for the congruent subgroup $\Gamma_0(M) \subset SL(2,\mathbb{Z})$ with trivial nebentypus and $M=p_1p_2$, where $p_2< p_1$ are two distinct primes.  Then for any $\epsilon > 0$, we have
		\begin{align*}
			L \left( {1}/{2}+it, \, \pi \times f \right) \ll_{\pi, \, \varepsilon}\, {\mathcal{Q}^{1/4}} \left(\frac{p_1}{p_2t^2}\right)^{3/40+\epsilon},
		\end{align*} 
		where $\mathcal{Q}=M^3 (1+|t|)^6$ denotes the  analytic conductor of the $L$-function $	L \left( {1}/{2}+it, \, \pi \times f \right) $.
	\end{theorem}
	\begin{remark}
		The above bound gives subconvexity in the range $p_2 < p_1 < p_2t^2$. Furthermore, fixing $p_1$ and $p_2$, we generalize  the main result  of Munshi \cite{r1} and Corollary 1.3 of  Lin--Sun \cite{r26}. 
	\end{remark}

	$L \left( {1}/{2}+it, \, \pi \times f \right) $ may be regarded as the $L$-value $	L \left( {1}/{2}, \, \pi \times F \right) $ at the  central point $1/2$ associated to $\pi$ and  $F=f \otimes |.|^{it}$ having varying level and varying archimedian parameters $( it, -it)$. Thus we also get subconvexity  for $L \left( {1}/{2}, \, \pi \times F \right) $ when the spectral parameters of the $GL(2)$ form $F$ varies in the range $p_1>p_2>p_1/t^2$, $t \gg p_1^{\epsilon}$.  In particular, we get subconvexity in the $GL(2)$ level aspect when the level $M \asymp p_1^2$ of  $F$ varies with the mild condition $t \asymp p_1^\epsilon$.  As was observed in \cite{SK}, in the delta method approach,   the  $GL(3) \times GL(2)$  structure does not distinguish between the $t$-aspect and $GL(2)$-spectral aspect. Indeed, the analytic conductor of  $F=f \otimes |.|^{it}$ and  $f$ (with varying spectral parameters)  remains the same.  Thus 
	our method  can also be  adapted to give subconvexity for $L(1/2, \pi \times f)$ in the range $p_2 < p_1 < p_2t_f^2$, where $(it_f,-it_f)$ are the spectral parameters of $f$,  generalising \cite{SK}  and also Huang \cite{BH}.

	For a vast introduction to the subconvexity problem and its importance, we refer the readers to \cite{PM}, \cite{ICM}.  In recent years, there has been tremendous progress on the subconvexity problem, most notably by Nelson \cite{PN} who resolved the problem for automorphic $L$-functions of any degree in the $t$-aspect and spectral aspect (`generic' cases). In the level aspect, the problem is  `largely' open for degree  $d \geq 3$ automorphic $L$-functions.  In the hybrid (level and $t$) aspect,  the first subconvexity result was proved by Heath-brown \cite{HB} for degree one $L$-functions $L(1/2+it, \chi)$ generalizing the classical  Burgess bound \cite{r7} and the Weyl bound \cite{r6}. In the $GL(2)$ case,  the influential work of Michel-Vankatesh \cite{MV} settles the subconvexity problem in all aspects.  In this article, we prove subconvexity for degree six  Rankin-Selberg $L$ functions in the hybrid $GL(2)$ level and $t$-aspect.\\
	
	\textbf{Comment on the method} We apply separation of oscillation technique coupled with `arithmetic' and `analytic' conductor lowering trick laid out by Munshi \cite{RM1}, \cite{s4} to prove Theorem   \ref{main theorem}.  Analysis of integral transforms forms the technical heart of our paper.  On comparing with Lin-Sun \cite{r26}, we observe that the presence of an extra integral at the beginning (analytic conductor lowering) does not obstruct us in obtaining the best possible exponent $3/20$ by the delta method. In \cite{KA} Aggarwal observed that the analytic conductor lowering is in-built in the delta method. Thus the presence of extra integral, in the beginning, is merely psychological. The $GL(3) \times GL(2)$ structure is crucial for the proof as the final character sum transforms into an additive character which gives us extra savings in the  `Poisson' step after Cauchy.   A contrasting difference between our proof and  Munshi and Lin-Sun's proof is the length of the `dual' $GL(2)$-sum which is larger than the `modulus' in our case.   We provide different arguments (by analyzing integral transforms separately) to get optimal savings in `the zero frequency'. Our method does not work for $p_1=p_2$ or $p_2=1$. New ideas are required to tackle these cases. In an upcoming work, Munshi has proved subconvexity for $L(1/2, \pi \times f)$ in the $GL(2)$ level aspect when the level $P$ is of the form $p^r$, $r\geq 3$. In an upcoming article, we apply his ideas to extend our result for  $p_1=p_2$.

	\textbf{Notations:} For  $z \in \mathbb{C}$, we set $e(z):= e^{2 \pi i z}$.  By $A \sim B$ we mean that $B \leq A \leq 2B$.  We use $X \ll Y$ to mean that there is a constant $K >0$ such that $|X| \leq K|Mt|^{\epsilon} Y $ and by  $A \asymp B$ we mean that $ A \ll B$ and $B \ll A$.  By the term `savings'  we mean  `initial bound'  / ` final bound'.

	\section{Sketch of proof}
	This section will provide a rough idea of the steps involved in proving our main result \eqref{main theorem}. By the approximate functional equation for the $L$-function $L\left(\frac{1}{2}+ it, \pi \times f \right)$, we see that our main object of study becomes the sum $S_r(N)$ given in  \eqref{a28}. We need to analyze this sum  $S_r(N)$ and show some cancellations to get our desired result. On applying the conductor lowering trick (see  \cite{RM1} and \cite{s4}) and a Fourier expansion of  delta symbol due to Duke-Friedlander-Iwaniec,  we arrive at
	\begin{align*}
		\frac{1}{KQp_1}&\displaystyle\int_{\mathbb{R}}\,\int_{\mathbb{R}}\, V_1\left(\frac{\nu}{K}\right)\, \displaystyle\sum_{1\leq q\leq Q}\frac{g(q,u)}{q}\sideset{}{^\star}{\displaystyle\sum}_{a \, \mathrm{mod} \, q}\,\,\sum_{b \, \mathrm{mod} \, p_1} \notag \\
		\times&\,\sum_{n=1}^{\infty}A_{\pi}(r, n)e\left(\frac{(a+bq)n}{p_1q }\right) \,n^{i\nu}\,e\left(\frac{nu}{p_1qQ}\right) V_1\left( \frac{n}{N}\right)\notag \\
		\times&\,\sum_{m=1}^{\infty}\lambda_f (m)\,e\left(-\frac{(a+bq)m}{p_1q }\right)\, m^{-i(t+\nu)}\,e\left(\frac{-mu}{p_1qQ}\right) V_2\left( \frac{m}{N}\right)\, \mathrm{d}u\,\mathrm{d}\nu,
	\end{align*}
	where $V_1$ and $V_2$ are the smooth and compactly supported functions, supported on $[1/2, 5/2]$.
	To capture the main ideas we consider the generic case $N = (p_1p_2)^{3/2}\,t^3$, $r =1$, $m \asymp N$, $n\asymp N$ and $q \sim Q = \sqrt{N/p_1K}$ with $t^{\epsilon} < K < t $. We get  $S_r(N) \ll N^2$ on estimating trivially.  
	Next, we apply the summation formulae to the sum over $m$ and $n$. On applying  the $GL(3)$ Voronoi summation formula to  the sum over $n$ we arrive at 
	\begin{align*}
		\frac{N^{2/3+i\nu}}{p_1q} \sum_{\pm}  \sum_{n_{2}\sim (p_1QK)^3/N}  \frac{A_{\pi}(n_{2},1)}{ n^{1/3}_{2}} S\left( \overline{(a+bq)}, \pm n_{2}; p_1q\right) \eth_1^{\pm}(...),
	\end{align*} where $ \eth_1^{\pm}(...)$ is an integral transform defined in Lemma \ref{a8}.  Assuming square-root cancellations in $ \eth_1^{\pm}(...)$ and applying Weil bound for the Kloosterman sum, we get   $\frac{N}{(p_1QK)^{3/2}}$ savings in this step.
	On applying the $GL(2)$ Voronoi summation formula to the sum over $m$ we arrive at
	\begin{align*}
		\frac{ N^{3/4-i(t+\nu)}}{(p_1q)^{1/2}\,p_2^{1/4}}  \sum_{\pm} \sum_{m \sim (p_1Qt)^2p_2/N} \frac{\lambda_f(m)}{m^{1/4}}  e\left(-m  \frac{\overline{(a+bq) p_2}}{q}\right) \mathcal{I}^{\pm} (...),
	\end{align*}  where $\mathcal{I}^{\pm} (...)$ is an integral transform defined in Lemma \ref{a9}. We save $N/p_1p_2^{1/2}Qt\,$ in this step. The summation formulae transform  $S_{r}(N)$ to the following expression with a negligible error term
	\begin{align*}
		\frac{ N^{17/12}}{KQ^{7/2}\,p^{5/2}_1p_2^{1/4}}\,\displaystyle\sum_{q\sim Q}\,\sum_{n_{2}\sim (p_1QK)^{3}/N}  \frac{A_{\pi}(n_{2},1)}{ n^{1/3}_{2}}\,\sum_{m \sim (p_1Qt)^{2}p_2/N} \frac{\lambda_f(m)}{m^{1/4}}\, \mathcal{C}(...)\,\mathcal{J}(...),
	\end{align*}
	where the character sum $\mathcal{C}(...)$ is given by
	\begin{align*}
		\sideset{}{^\star}{\displaystyle\sum}_{a \, \mathrm{mod} \, q}\,\,\,\sum_{b \, \mathrm{mod} \, p_1}\,S\left( \overline{(a+bq)}, \pm n_{2}; p_1 q\right)\,e\left(-m  \frac{\overline{(a+bq) p_2}}{p_1q}\right)  \rightsquigarrow \, p_1q\, e\left(\pm \frac{\bar{m}p_2n_2}{p_1q}\right).
	\end{align*}  Thus we  save $\sqrt{p_1Q}$ in the $a$ and $b$-sum.  In the $\nu$-integral, we get the square-root cancellation of size $\sqrt{K}$ by stationary phase analysis.  Our total savings so far is
	\begin{align*}
		\frac{N}{(p_1QK)^{3/2}} \times \frac{N}{p_1p^{1/2}_2Qt} \times \sqrt{p_1Q}\times \sqrt{K}   =\frac{N}{p_1\,p_2^{1/2}\,t} . 
	\end{align*}
	Thus, we need to save $p_1\,p_2 ^{1/2}\,t$ and a little extra to go beyond the trivial bound. To get these savings, we apply  Cauchy-Schwartz inequality to the  sum over $n_2$ to get rid of $A_{\pi}(1,n)$ and we arrive at
	\begin{align*}
		\frac{ N^{5/12}\,}{Q^{3}K^{1/2}\,p^2_1p_2^{1/4}}\,\Bigg(\sum_{n_{2}\sim (p_1QK)^3/N} \,\left|\displaystyle\sum_{q\sim Q}\,\,\sum_{m \sim (p_1Qt)^2p_2/N} \frac{\lambda_f(m)}{m^{1/4}}\,e\left(\frac{\bar{m}p_2n_2}{p_1q}\right)\, \mathcal{J}(...)\right|^2 \Bigg)^{1/2}.
	\end{align*}
	Opening the absolute valued square and applying the Poisson summation formula to sum over $n_2$, we see that the sum over $n_2$ gets transformed into the following expression
	\begin{align*}
		\sum_{n_{2}}\,\displaystyle\sum_{q\sim Q}\,\displaystyle\sum_{q'\sim Q}\,\,\sum_{m \sim (p_1Qt)^2p_2/N}\,\sum_{m' \sim (p_1Qt)^2p_2/N} \lambda_f(m)\lambda_{f}(m^\prime) \,\mathfrak{C}(...)\,\mathcal{G}(...),
	\end{align*} where $\mathfrak{C}(...)$ is a character sum defined in  \eqref{a20} and the integral transform $\mathcal{G}(...)$ is defined in  \eqref{a26}.  In the case of zero frequency $n_2 =0$, we will save the whole diagonal length i.e. $p^2_1Q^3 t^2p_2/N$ which is sufficient if 
	\begin{align*}
		\frac{p^2_1Q^3 t^2p_2}{N} > p^2_1\,p_2\,t^2 \hspace{0.3cm} \Longleftrightarrow \hspace{0.3cm} K <  \left(\frac{p_2t^2}{p_1}\right)^{3/4}.
	\end{align*} In the case of non-zero frequency  $n_2 \neq 0$, we will save $	{(p_1QK)^{3}}/{N\sqrt{K}}$,  which is satisfactory when
	\begin{align*}
		\frac{(p_1QK)^{3}}{N\sqrt{K}} > p^2_1\,p_2\,t^2  \hspace{0.3cm} \Longleftrightarrow \hspace{0.3cm} K > \,\left(\frac{p_2t^2}{p_1}\right)^{1/4}.
	\end{align*}
	
	The optimal choice for  $K$ is  obtained by equating both the savings and  we get $K = \left({p_2t^2}/{p_1}\right)^{2/5}$.  Thus we obtain the range $p_2 < p_1 < p_2t^2$.

	\section{Preliminaries} 
	\subsection{Delta Method}\label{Delta}
	In this paper, we use the delta symbol expansion due to Duke, Friedlander, and Iwaniec. More specifically, we use the expansion $(20.157)$ given in Chapter 20 of \cite{ana}. Let $\delta : \mathbb{Z}\to \{0,1\}$ be defined by
	\[
	\delta(n,m)=
	\begin{cases}
		1 &\text{if}\,\,n=m \\
		0 &\text{otherwise}.
	\end{cases}
	\]
	For $n,m\in\mathbb{Z}\cap [-2L,2L]$, we have
	\begin{equation}\label{a1}
		\delta(n,m)=\frac{1}{Q} \sum_{1 \leq q \leq Q} \frac{1}{q} \, \ \sideset{}{^\star} \sum_{a\bmod q}e\left(\frac{(n-m) a}{q}\right)\int_{\mathbb{R}} \psi(q,x)e\left(\frac{(n-m) x}{q	Q}\right) d x, 
	\end{equation} where $Q=2L^{1/2}$ and   $\psi(q,u)$  satisfies the following properties  (see $(20.158)$ and $(20.159)$ of \cite{ana} and \cite[Lemma 15]{r24})
	
	\begin{align}\label{delta}
		&\psi(q,x)=1+h(q,x),\,\,\,\,\text{with}\,\,\,\,h(q,x)=O\left(\frac{Q}{q }\left(\frac{q}{Q}+|x|\right)^A\right)\\
		&\psi(q,x)\ll |x|^{-A}, \ \ A >1\\
		&x^j \frac{\partial^j}{ \partial x^j} \psi(q, x) \ll \  \min \left\lbrace \frac{Q}{q}, \frac{1}{|x|} \right\rbrace \  \log Q, \ \ j \geq 0.
	\end{align}
	\vspace{0.2cm}
	In particular, the second property implies that the effective range of the  $x$-integral in \eqref{a1} is given by  $[-L^{\epsilon}, L^{\epsilon}]$. It also follows that if $q \ll Q^{1- \epsilon}$ and $  x \ll Q^{- \epsilon} $, then  $ \psi(q, x)$ can
	be replaced by $1$ at a cost of a negligible error term. If $ q \gg Q^{1- \epsilon}$, then we get $ x^j \frac{\partial^j}{ \partial x^j} \psi(q, x) \ll Q^{\epsilon}$, for any $ j \geqslant 1$. If $ q \ll  Q^{1- \epsilon}$  and $  Q^{- \epsilon} \ll |x| \ll   Q^{ \epsilon}$, then $ x^j \frac{\partial^j}{ \partial x^j} \psi(q, x) \ll Q^{\epsilon}$, for any $ j \geqslant 1$. 
	
	\subsection{Voronoi  formula for $GL(3)$ } \label{gl3 maass form} 
	
	Let $\pi$ be a Maass form of type $(\nu_{1}, \nu_{2})$ for $SL(3, \mathbb{Z})$.  Langlands parameters $({\bf \alpha}_{1},{\bf \alpha}_{2},{\bf \alpha}_{3})$ of $\pi$ are given by
	$${\bf \alpha}_{1} = - \nu_{1} - 2 \nu_{2}+1, \, {\bf \alpha}_{2} = - \nu_{1}+ \nu_{2},  \, {\bf \alpha}_{3} = 2 \nu_{1}+ \nu_{2}-1.$$
	For $\psi \in C_{c}^{\infty}(0, \infty)$, we denote $\tilde \psi(s) = \int_{0}^{\infty} \psi(x) x^{s-1} \mathrm{d}x$ to be its Mellin transform. For $\sigma > -1 + \max \{-\Re({\bf \alpha}_{1}), -\Re({\bf \alpha}_{2}), \Re({\bf \alpha}_{3})\}$ and $\ell =0$ and $1$, we define
	\begin{equation} \label{A15}
		G_{\ell}(y) :=  \frac{1}{2\pi i} \int_{(\sigma)} (\pi^3 y)^{-s} \, \prod_{i=1}^{3} \frac{\Gamma\left(\frac{1+s+{\bf \alpha}_{i}+ \ell}{2}\right)}{\Gamma\left(\frac{-s-{\bf \alpha}_{i}+ \ell}{2}\right)} \tilde \psi(-s) \mathrm{d}s. 
	\end{equation}
	
	Set \[G_{\pm}(y) = \frac{1}{2\pi^{3/2}} \big(G_0(y) \mp iG_1(y) \big). \]
	\begin{lemma} \label{gl3voronoi}
		Let $\psi (x)$ be a compactly supported smooth function on $(0,\infty)$. Let $A_{\pi}(m,n)$ denote the $(m,n)$-th Fourier coefficients of the Maass form $\pi$. Then we have
		\begin{align} \label{GL3-Voro}
			& \sum_{n=1}^{\infty} A_{\pi}(m,n) e\left(\frac{an}{q}\right) \psi(n) \\
			\nonumber & =q  \sum_{\pm} \sum_{n_{1}|qm} \sum_{n_{2}=1}^{\infty}  \frac{A_{\pi}(n_{2},n_{1})}{n_{1} n_{2}} S\left(m \bar{a}, \pm n_{2}; mq/n_{1}\right) \, G_{\pm} \left(\frac{n_{1}^2 n_{2}}{q^3 m}\right)
		\end{align} 
		where $(a,q)=1, \bar{a}$ is the multiplicative inverse modulo $q$ and $$S(a,b;q) = \sideset{}{^\star}{\sum}_{x \,\rm mod \, q} e\left(\frac{ax+b\bar{x}}{q}\right) $$
		is the Kloostermann sum.
	\end{lemma}
	\begin{proof}
		See \cite{r2}.
	\end{proof}
	The following lemma gives an asymptotic expansion for $G_{\pm}(x)$.
	\begin{lemma} \label{GL3oscilation}
		Let $G_{\pm}(x)$ be as above,  and  $g(x) \in C_c^{\infty}(X,2X)$. Then for any fixed integer $K \geq 1$ and $xX \gg 1$, we have
		\begin{equation*}
			G_{\pm}(x)=  x \int_{0}^{\infty} g(y) \sum_{j=1}^{K} \frac{c_{j}({\pm}) e\left(3 (xy)^{1/3} \right) + d_{j}({\pm}) e\left(-3 (xy)^{1/3} \right)}{\left( xy\right)^{j/3}} \, \mathrm{d} y + O \left((xX)^{\frac{-K+5}{3}}\right),
		\end{equation*}
		where $c_{j}(\pm)$ and $d_{j}(\pm)$ are some  absolute constants depending on $\alpha_{i}$,  $i=1,\, 2,\, 3$.  
	\end{lemma}
	\begin{proof}
		See  \cite{r2}.
	\end{proof}
	
	The following well-known lemma gives the Ramanujan bound for the $GL(3)$ Fourier coefficients $A_{\pi}(m,n)$ on average.	
	\begin{lemma} \label{ramanubound}
		We have 
		$$\mathop{\sum \sum}_{n_{1}^{2} n_{2} \leq X} \vert A_\pi(n_{1},n_{2})\vert ^{2} \ll \, X^{1+\epsilon}.$$
	\end{lemma}
	\begin{proof}
		See \cite{XLI}. 
	\end{proof}

	\subsection{$GL(2)$  Voronoi summation formula} 
	\begin{lemma} \label{a4}
		Let $f\in S_k(D, \chi_{D})$  be a Hecke cusp form with Fourier coefficients $\lambda_f(n)$ and trivial Nebentypus $\chi_D$.   Let $a$  and  $c>0$ be coprime integers.  Let $D=D_1D_2$ with   $D_1 = (c, D)$ and $(D_1,D_2)=1$.   Let $F$ be a compactly supported smooth function supported on $[1,2]$. Then  we have 
		\begin{equation} 
			\sum_{m=1}^\infty \lambda_f (m) e\left( \frac{am}{c}\right) F(m) = \frac{1}{c}   \frac{\eta_f(D_2)}{\sqrt{D_2}} \sum_{m=1}^\infty \lambda_{f}( m) e\left(  \frac{-m\overline{a D_2}}{c}\right) H\left( \frac{m}{D_2 c^2}\right),
		\end{equation}
		where $ a \overline{a} \equiv 1 (\textrm{mod} \  c)$, $|\eta_f(D_2)|=1$ and 
		\begin{align*} 
			H (y)=   2 \pi i^k \int_0^\infty F(x)  J_{k-1} \left( 4\pi \sqrt{xy}\right) dx.
		\end{align*} 
	\end{lemma} 
	\begin{proof}
		See \cite{r3}.   
	\end{proof}
	We will  use some results related to the estimation of the exponential integral of the form: 
	\begin{equation} \label{eintegral}
		\mathfrak{I}= \int_a^b g(t) e(f(t)) dt,
	\end{equation} where $f$ and $g$ are  real valued smooth functions on the interval $[a, b]$. 
	
	\begin{lemma} \label{sdb}
		Let $f$ and $g$ be real valued twice differentiable functions and let $f^{\prime \prime} \geq r>0$ or  $f^{\prime \prime} \leq -r <0$, throughout the interval $[a, b]$. Let $g(x)/f^\prime(x)$ is monotonic and $|g(x)| \leq A$. Then we have
		
		\begin{align*}
			|\mathfrak{I}|\, \leq\, \frac{8A}{\sqrt{r}}. 
		\end{align*} 
	\end{lemma} 
	\begin{proof}
		See \cite[Lemma 4.5, page 72]{r4}. 
	\end{proof}

	\section{Set-up}
	
	\subsection{Approximate functional equation} 
	Let $\pi$ and $f$ be defined as in Theorem \ref{main theorem}. An application of the Functional equation of $L(1/2+it, \pi \times f)$  gives the following lemma.
	
	\begin{lemma}\label{AFE}
		Let $\mathcal{Q}=M^3 (1+|t|)^6$ be the analytic   conductor of  $L \left(1/2+it, \pi \times f \right)$. Then, as $\mathcal{Q} \rightarrow \infty$, we have
		\begin{align} \label{aproxi} 
			L \left( \frac{1}{2}+it, \, \pi \times f \right) \ll_{\pi, \, \varepsilon} \mathcal{Q}^{\varepsilon}\sum_{r \leq \mathcal{Q}^{(1+2\varepsilon)/4}}\frac{1}{r}  \sup_{ N\leq \frac{\mathcal{Q}^{1/2+\varepsilon}}{r^2}}\frac{|S_r(N)|}{N^{1/2}} +\mathcal{Q}^{-2023},
		\end{align}
		where
		\begin{align} \label{a28}
			S_r(N) = \mathop{\sum }_{n=1}^{\infty} A_{\pi}(r, n) \lambda_f(n)n^{-it}  V_1 \left(\frac{n}{N}\right),
		\end{align}
		for some smooth function $V_1$ supported in $[1,2]$ and  satisfying $V_1^{(j)}(x) \ll_{j} 1$.
	\end{lemma} 
	\begin{proof}
		See \cite[Theorem 5.3]{ana}.  
	\end{proof}
	\subsection{Applying DFI delta method} 
	Following Munshi's method \cite{r1}, our first step is to separate  the oscillations of $A_{\pi}(r, n)$ and $\lambda_{f}(n)\,n^{-it}$  using DFI delta  method \eqref{a1}. To this end we  write $S_r(N)$ as
	
	\begin{align*}
		S_r(N)&=\mathop{\mathop{\sum\sum}_{\substack{n, \,m =1 }}^{\infty}} \delta{(n, m)}\,A_{\pi}(r, n)m^{-it}\,\lambda_f(m)\,V_1\left(\frac{n}{N}\right)\,V_2\left(\frac{m}{N}\right),
	\end{align*}
	where $V_2$ is a smooth function supported on $[1/2,5/2]$ with $V_2^{(j)}(x) \ll_j 1$.
	Next, we apply arithmetic and analytic conductor lowering trick following Munshi  \cite{RM1}, \cite{s4} by introducing an integral and a congruence condition.   Thus we see that  
	\begin{align*}
		S_r(N) =	\frac{1}{K}\,\int_{\mathbb{R}}\,V_1\left(\frac{\nu}{K}\right) \mathop{\mathop{\sum\sum}_{\substack{n, \,m =1 }}^{\infty}}\delta{(n, m)}\, A_{\pi}(r, n)m^{-it}\,\lambda_f(m)\,\left(\frac{n}{m}\right)^{i\nu}\,V_1\left(\frac{n}{N}\right)\,V_2\left(\frac{m}{N}\right)\,d\nu,
	\end{align*}
	where $t^{\epsilon} < K < t^{1-\epsilon}$ is a parameter which  will be determined optimally later.  On applying integration by parts to the $\nu$-integral,  we see that it is negligibly small unless $|n-m| \ll N/K$.   We further  detect the equation $n=m$ as follows
	\begin{align*}
		n =m \,\, \iff\,\, n \equiv m \ \mathrm{mod} \ p_1,   \,\, \frac{n-m}{p_1} =0.
	\end{align*}
	On applying  the Fourier expansion   \eqref{a1} of $\delta{(n, m)}$ with $Q = \sqrt{N/p_1K}$. We see that,  $S_r(N)$, up to a negligible error term, is given by
	\begin{align}\label{a7}
		\notag&\,\frac{1}{KQp_1}\displaystyle\int_{\mathbb{R}} W(u)\,\int_{\mathbb{R}}\, V_1\left(\frac{\nu}{K}\right)\, \displaystyle\sum_{1\leq q\leq Q}\frac{\psi(q,u)}{q} \sideset{}{^\star}{\displaystyle\sum}_{a \, \mathrm{mod} \, q}\,\,\sum_{b \, \mathrm{mod} \, p_1}\\
		&\times \sum_{n=1}^{\infty}A_{\pi}(r,n)e\left(\frac{(a+bq)n}{p_1q }\right) \,v_1(n)  \sum_{m=1}^{\infty}\lambda_f (m)\,e\left(-\frac{(a+bq)m}{p_1q }\right)\,v_2(m)\, \mathrm{d}u\,\mathrm{d}\nu,
	\end{align}
	where $W$ is a compactly supported smooth function supported on $[-2X^{\epsilon}, 2X^{\epsilon}]$ with $W(u) = 1$ on $[-X^{\epsilon}, X^{\epsilon}]$, satisfying $W^{(j)}(u) \ll_j 1$ and
	\begin{align}\label{a100}
		v_1(n) = n^{i\nu}\,e\left(\frac{nu}{p_1qQ}\right) V_1\left( \frac{n}{N}\right),  \hspace{0.5cm} v_2(m) =   m^{-i(t+\nu)}\,e\left(\frac{-mu}{p_1qQ}\right) V_2\left( \frac{m}{N}\right).
	\end{align}
	
	We consider  the following  cases:
	\begin{align*}
		&\textbf{Case 1}:\,\,  (p_1, q) =1\,\, \text{and}\,\, b \equiv 0 \bmod p_1\\
		&\textbf{Case 2}:\,\,  (p_1, q) =1\,\, \text{and}\,\, b \not\equiv 0 \bmod p_1\\
		&\textbf{Case 3}:\,\,  (p_1, q) \neq 1\,.
	\end{align*} 
	In Case 2, we make a change of variable $a \rightarrow ap_1$. Thus we combine the sum over $a$ and $b$ as follows: 
	\[ \sideset{}{^\star}{\displaystyle\sum}_{a \, \mathrm{mod} \, q}  \,\, \, \sideset{}{^\star}{\displaystyle\sum}_{b \, \mathrm{mod} \, p_1} \cdots\,=\, \sideset{}{^\star}{\displaystyle\sum}_{c \, \mathrm{mod} \, qp_1} \cdots  \]
	In Case 3, we write  $q = p^{\ell}_1 q'$ with $\ell >  0$ and $(p_1, q') = 1$, and we decompose the sum over $a$ as follows: 
	\[\sideset{}{^\star}{\displaystyle\sum}_{a \, \mathrm{mod} \, q} \sum_{b \, \mathrm{mod} \, p_1}\cdots=\sideset{}{^\star}{\displaystyle\sum}_{a_1 \, \mathrm{mod} \, q'}\,  \sideset{}{^\star}{\displaystyle\sum}_{a_2 \, \mathrm{mod} \, p_1^\ell} \, \sum_{b \, \mathrm{mod} \, p_1}\cdots \]
	Accordingly we write $a+bq=a_1 p^{\ell}_1 + a_2q'+b p^{\ell}_1 q'=a_1p^{\ell}_1+(a_2+bp^{\ell}_1) q'$. We combine the sum over $a_2$ and $b$ to form a new variable $u=a_2+bp^\ell$, where $u$ runs over all the classes mod $p^{1+\ell}$ coprime to $p_1$.  Thus, we see that 
	\[\sideset{}{^\star}{\displaystyle\sum}_{a \, \mathrm{mod} \, q} \sum_{b \, \mathrm{mod} \, p_1}\cdots=\sideset{}{^\star}{\displaystyle\sum}_{a_1 \, \mathrm{mod} \, q'}\,  \sideset{}{^\star}{\displaystyle\sum}_{u \, \mathrm{mod} \, p_1^{1+\ell}}  \cdots = \sideset{}{^\star}{\displaystyle\sum}_{c \, \mathrm{mod} \, q'p_1^{1+\ell}}  \cdots\] 
	Thus we arrive at the following expression: 
	\begin{align*}
		S_r(N) \,=\, \mathcal{S}_1 + \mathcal{S}_2,
	\end{align*}
	where 
	\begin{align}\label{hb1}
		\notag \mathcal{S}_1 =& \,\frac{1}{KQp_1}\displaystyle\int_{\mathbb{R}} W(u)\,\int_{\mathbb{R}}\, V_1\left(\frac{\nu}{K}\right)\, \displaystyle\sum_{1\leq q\leq Q}\frac{\psi(q,u)}{q}\sideset{}{^\star}{\displaystyle\sum}_{a \, \mathrm{mod} \, q}\\
		&\times \, \sum_{n=1}^{\infty}A_{\pi}(r,n)e\left(\frac{an}{q }\right) \,v_1(n)\,\sum_{m=1}^{\infty}\lambda_f (m)\,e\left(-\frac{am}{q }\right)\,v_2(m)\, \mathrm{d}u\,\mathrm{d}\nu,
	\end{align}
	and 
	\begin{align}\label{hb2}
		\notag \mathcal{S}_2 =& \,\frac{1}{KQp_1}\displaystyle\int_{\mathbb{R}} W(u)\,\int_{\mathbb{R}}\, V_1\left(\frac{\nu}{K}\right)\, \displaystyle\sum_{1\leq q\leq Q}\frac{\psi(q,u)}{q} \sideset{}{^\star}{\displaystyle\sum}_{c \, \mathrm{mod} \, p_1^{1+\ell}q' }\,\\
		&\times \, \sum_{n=1}^{\infty}A_{\pi}(r,n)e\left(\frac{c n}{p_1q }\right) \,v_1(n)\,\sum_{m=1}^{\infty}\lambda_f (m)\,e\left(-\frac{cm}{p_1q }\right)\,v_2(m)\, \mathrm{d}u\,\mathrm{d}\nu,
	\end{align}
	where $q=q'p^\ell$, $\ell \geq 0$. 
	We will analyse $ \mathcal{S}_1$ in Section \ref{S1}.  Next, we focus on $ \mathcal{S}_2$. 
		\section{Application of  summation formulae}  In this section, we will apply  $GL(2)$ and $GL(3)$-Voronoi summation formulae to the sum over $m$ and $n$ respectively.
		
		Let
		\begin{align}\label{ns}
			\mathcal{Z}_1 = \, \sum_{n=1}^{\infty}A_{\pi}(r,n)e\left(\frac{cn }{p_1q }\right) \,v_1(n),    
		\end{align}
		
		\begin{align}\label{ms}
			\mathcal{Z}_2 = \, \sum_{m=1}^{\infty}\lambda_f (m)\,e\left(-\frac{c m}{p_1q }\right)\,v_2(n) ,    
		\end{align}
		where  $v_1(n)$ and $v_2(m)$ are as in  \eqref{a100}.

		\begin{lemma}\label{a8}
			Let $\mathcal{Z}_1$ be as in \eqref{ns}. Then we have 
			\begin{align*}
				\mathcal{Z}_1 = \frac{N^{2/3+i\nu}}{p_1qr^{2/3}} \sum_{\pm} \sum_{n_{1}|p_1qr}\,n^{1/3}_1 \sum_{n_{2}\ll N_0/n^2_1}  \frac{A_{\pi}(n_{2},n_{1})}{ n^{1/3}_{2}} \, S\left( r\overline{c}, \pm n_{2}; p_1qr/n_{1}\right)  \mathcal{I}^{\pm}_1&(n^2_1n_2, u; q) \\
				&+ O(N^{-2024}),
			\end{align*} where
			$$N_0 =  \,{\max}\left\{\frac{(p_1qK)^3r}{N} , N^{1/2} (p_1K)^{3/2} r\right\}N^{\epsilon},$$ and 
			\begin{align*}
				\mathcal{I}^{\pm}_1(n^2_1n_2, u; q) =  \int_{0}^\infty U_1(z) z^{i\nu}\, e\left(\frac{Nuz}{p_1qQ} \pm  \frac{3(Nz n_1^2 n_2)^{1/3}}{p_1q r^{1/3}} \right)dz ,
			\end{align*}
			with $U_1$  a new smooth bump function. 
		\end{lemma}
		\begin{proof}
			On applying the $GL(3)$ Voronoi formula  (see  Lemma \ref{gl3voronoi})  along with Lemma \ref{GL3oscilation}  to $\mathcal{Z}_1$ gives the lemma. Note that we take $K$ sufficiently large enough in Lemma a \ref{GL3oscilation} so that the error term is negligibly small. We work with $j=1$ as the analysis for other terms is similar and gives even better result.  Indeed  for $j \geq 2$, we still have the same phase function $e(3(xy)^{1/3})$ however we get the extra $x^{(j-1)/3}$ factor in the denominator which gives extra savings.  The dual $n_2$-length $n_1^2n_2 \ll N_0$ is obtained by applying integration by parts to $\mathcal{I}^{\pm}_1(n^2_1n_2, u; q)$. 
		\end{proof}
		Next we apply  $GL(2)$-Voronoi  formula  to $\mathcal{Z}_2$.
		\begin{lemma}\label{a9}
			Let $\mathcal{Z}_2$ be as  in  \eqref{ms}. Then we have 
			\begin{align*}
				\mathcal{Z}_2 = \frac{ N^{3/4-i(t+\nu)}\,\eta_f(D_2)}{(p_1q)^{1/2}\,D_2^{1/4}}  \sum_{\pm} \sum_{m \ll M_0} \frac{\lambda_f(m)}{m^{1/4}}  e\left(-m  \frac{\overline{c D_2}}{p_1q}\right) \mathcal{I}^{\pm}_2(m,u; q) + O(N^{-2024}),
			\end{align*} where $$M_0 = \,{\max}\left\{\frac{(p_1q t)^{2}D_2}{N }, \,  p_1D_2K\right\}N^{\epsilon},$$ 
			$D_2=p_2/(q',p_2)$, $q=p^\ell q'$ 
			and 
			\begin{align}\label{l1}
				\mathcal{I}^{\pm}_2(m,u; q)=	 \,\int_0^ \infty U_2(y)\,y^{-i(t+\nu)}\, e\left(\frac{-N uy}{p_1qQ} \pm \frac{2\sqrt{mNy}}{p_1q \sqrt{D_2}} \right)dy,
			\end{align}
			with $U_2$ being a new smooth bump function. 
		\end{lemma}
		\begin{proof}
			On applying the $GL(2)$ Voronoi summation formula (see Lemma \ref{a4}) with $c = p_1q$ and $D = p_1p_2 = D_1 D_2$, where $D_1 = (p_1q, p_1p_2)$, $D_2 = D/D_1=p_2/(p_2, q)$ to  $\mathcal{Z}_2$, we obtain
			\begin{align}\label{l3}
				\mathcal{Z}_2 &= \frac{1}{p_1q}   \frac{\eta_f(D_2)}{\sqrt{D_2}} \sum_{m=1}^\infty \lambda_{f}( m) e\left(-m  \frac{\overline{L_{p_1, q} D_2}}{p_1q}\right) \mathcal{V}_2\left( \frac{m}{D_2 (p_1q)^2}\right), 
			\end{align} 
			where 
			\begin{align*}
				\mathcal{V}_2\left( \frac{m}{D_2(p_1q)^2}\right) = 2\pi i^k \int_0^ \infty y^{-i(t+\nu)}\,e\left(\frac{-yu}{p_1qQ}\right) V_2\left( \frac{y}{N}\right)    J_{k-1} \left(  \frac{4\pi}{p_1q}\sqrt{\frac{my}{D_2}}  \right) \ { d} y. 
			\end{align*} 
			%
			%
			Changing the variable $y \longrightarrow Ny$, and extracting the oscillations of the Bessel function using 
			
			$$ J_{k-1} (4 \pi x) = \frac{1}{\sqrt{4\pi x}}\left(e(2x)\,U_{k-1} (4\pi x)  + e(-2x)\,\overline{U_{k-1}} (4\pi x)\right),$$ where $U_{k-1}$ is a smooth function with  $$x^{j} U_{k-1}^{(j)}(x) \ll_{j} \frac{1}{\sqrt{x}}, \ \ j \geq 0, \ x\gg 1,$$ 
			we get 
			\begin{align}\label{l2}
				\mathcal{V}_2\left( \frac{m}{D_2(p_1q)^2}\right) = \,\frac{D_2^{1/4}(p_1q)^{1/2}}{m^{1/4}} \frac{N^{3/4}}{N^{i(t+\nu)}} \,\sum_{\pm} \mathcal{I}^{\pm}_2(m,u; q),  
			\end{align}
			where  $\mathcal{I}^{\pm}_2(m,u; q)$ is defined in  \eqref{l1}. The dual length $m \ll M_0$ is determined by applying integration by parts to $ \mathcal{I}^{\pm}_2 (m,u; q)$. By plugging the value of $\mathcal{V}_2\left( \frac{m}{D_2(p_1q)^2}\right)$ from  \eqref{l2} to  \eqref{l3} we get the lemma.
		\end{proof}
		We conclude this section by recording the analysis done so far in the next lemma. 
		
		\begin{lemma}\label{a10}
			We have
			\begin{align*}
				\notag \mathcal{S}_2   =& \frac{ N^{17/12 -it}\,\eta_f(D_2)}{KQ\,p_1^{5/2}\,D_2^{1/4}\,r^{2/3}}\,\displaystyle\sum_{1\leq q\leq Q}\,\frac{1}{q^{5/2}}\sum_{\pm}\sum_{n_{1}|p_1q r}\,n^{1/3}_1 \sum_{n_{2}\ll N_0/n^2_1}  \frac{A_{\pi}(n_{2},n_{1})}{ n^{1/3}_{2}}\sum_{\pm}\sum_{m \ll M_0} \frac{\lambda_f(m)}{m^{1/4}}\\
				&\times\,\, \mathcal{C}^{\pm}_2(n^2_1n_2,m; q)\,\mathcal{J}^{\pm,\pm}_2(n^2_1n_2,m; q),
			\end{align*}
			where
			\begin{align*}
				\mathcal{J}^{\pm,\pm}_2(n^2_1n_2,m; q) := \displaystyle\int_{\mathbb{R}}\,V_1\left(\frac{\nu}{K}\right)\, \displaystyle\int_{\mathbb{R}} W(u)\,\psi(u,q)\,	\mathcal{I}^{\pm}_1(n^2_1n_2,u; q)\,\mathcal{I}^{\pm}_2(m,u; q)\,du\,d\nu,
			\end{align*}
			and   
			\begin{align*}
				\mathcal{C}^{\pm}_2(n^2_1n_2,m; q) := \sideset{}{^\star}{\displaystyle\sum}_{c \, \mathrm{mod} \, p^{1+\ell}q'}\,\,S\left( r\overline{c}, \pm n_{2}; p_1 qr/n_{1}\right)\,e\left(-m  \frac{\overline{c D_2}}{p_1q}\right).
			\end{align*} 
		\end{lemma}

		%
		%
		\vspace{0.3cm}	
		\section{Simplification of Integrals}
		On plugging the expression of  $	\mathcal{I}^{\pm}_1(n^2_1n_2, u; q)$ and $\mathcal{I}^{\pm}_2(m, u ; q)$ from   Lemma \ref{a8} and  Lemma \ref{a9} in $\mathcal{J}^{\pm,\pm}_2(n^2_1n_2,m; q)$ ,  we get 
		\begin{align}\label{a12}
			\mathcal{J}^{\pm,\pm}_2(n^2_1n_2,m; q) =& \, \displaystyle\int_{\mathbb{R}}\,V_1\left(\frac{\nu}{K}\right)\,\int_{\mathbb{R}}\, W(u)\,\psi(u,q)\\\
			&\times \int_0^ \infty U_2(y)\, y^{-i(t+\nu)}  e\left(\frac{-N uy}{p_1qQ} \pm \frac{2\sqrt{mNy}}{p_1q \sqrt{D_2}} \right) \notag\\ 
			&\times \int_{0}^\infty U_1(z)\, z^{i\nu} \, e\left(\frac{Nuz}{p_1qQ} \pm  \frac{3(Nz n_1^2 n_2)^{1/3}}{p_1q r^{1/3}} \right)du\,dy\,dz\,d\nu \notag.
		\end{align}
		Consider the $u$-integral 
		\begin{align}\label{A35}
			\int_{\mathbb{R}}W(u)\, \psi(q,u)\,e\left(\frac{(z-y)Nu}{p_1qQ} \right) du.
		\end{align}
		Let $q \ll Q^{1-\epsilon}$. In this case, we split the $u$-integral into two parts  $|u| \ll Q^{-\epsilon}$ and $|u| \gg Q^{-\epsilon}$. By using property \eqref{delta} of $\psi(q,u)$, for $|u| \ll Q^{-\epsilon}$, we can replace $\psi(q,u)$ by 1 with  negligible error term.  Thus, we arrive at 
		\begin{align*}
			\int_{|u| \ll Q^{-\epsilon}}W(u)\, e\left(\frac{(z-y)Nu}{p_1qQ} \right) du.
		\end{align*}
		On applying the integration by parts repeatedly, we see  that the  above integral is negligibly small unless
		\begin{align}\label{A36}
			|z-y| \ll \frac{p_1qQ}{N}Q^{\epsilon} = \frac{q}{QK}Q^{\epsilon}.
		\end{align}
		For $|u| \gg Q^{-\epsilon}$, using  integration by parts   and 
		\begin{align*}
			\frac{\partial^j}{ \partial u^j} \psi(q, u) \ll \  \min \left\lbrace \frac{Q}{q}, \frac{1}{|u|} \right\rbrace \  \frac{\log Q}{|u|^j}\ll Q^{\epsilon j}, \hspace{0.5cm}	W^j(u) \ll Q^{\epsilon j},
		\end{align*} 
		we get the restriction \eqref{A36} upto a negligible error term. 
		Now it remains to consider the generic case $q\gg Q^{1-\epsilon}$.  In this case we can use the $\nu$-integral to conclude that 
			\begin{align}\label{--}
			|z-y| \ll \frac{1}{K}Q^{\epsilon} \ll  \frac{q}{QK}Q^{\epsilon}.
		\end{align}


		Writing $z- y = w$ with $|w| \ll {q}Q^{\epsilon}/{QK}$, we get at the following expression for  the $y$-integral 
		\begin{align}\label{a13}
			\mathfrak{J}^{\pm,\pm}_{2}(n_1^2n_2,m; q)=	\int_0^ \infty U(y)\,y^{-it} \, e\left(\pm \frac{2\sqrt{mNy}}{p_1q \sqrt{D_2}} \pm\frac{3(N(y+w) n_1^2 n_2)^{1/3}}{p_1qr^{1/3}} \right) \,dy\
		\end{align} 
		where \[U(y) =  d{\pi}^{7/2}i^k\,U_2(y)U_1(y+w)\,\left(1+\frac{w}{y}\right)^{i\nu}\] is a new  weight function satisfying $U^{(j)} (y) \ll_j 1$.  Thus we arrive at
		\begin{align*}
			\mathcal{J}^{\pm,\pm}_2(n^2_1n_2,m; q) &=\displaystyle\int_{\mathbb{R}}\,V_1\left(\frac{\nu}{K}\right)\,\int_{\mathbb{R}}\, W(u)\,\psi(u,q)  \\
		&	 \int_{w \ll \frac{qQ^{\epsilon}}{QK}}e\left(\frac{N uw}{p_1qQ}\right)	\mathfrak{J}^{\pm,\pm}_{2}(n_1^2n_2,m; q) dw\,  du\, d \nu+O(Q^{-A})
		\end{align*}
	 Following Munshi  \cite{r1}, we will  estimate the $w$, $u$ and $\nu$-integral trivially later.  Indeed we estimate the $u$-integral as follows: 
	 \[ \int W(u) \psi(u,q)e\left(\frac{N uw}{p_1qQ}\right) du  \ll Q^\epsilon,\]
	 where we use $\int|\psi(u,q)|^2 du \ll Q^\epsilon$ (see \cite[p.1548]{r1}). We estimate the $\nu$ and $w$-integral trivially. We will further focus  on $ \mathfrak{J}^{\pm,\pm}_{2}(n_1^2n_2,m; q)$ only and will work with this integral only. 
	We use the size of the upper bounds for  $\nu$ and $w$-integral which are $K$ and $q/QK$ respectively from now on. 
		%
		%
		%
		\section{Cauchy-Schwarz and Poisson}\label{sec7}
		In this section, we apply the Cauchy-Schwarz inequality to the  $ n_2 $-sum in $\mathcal{S}_2$  followed by the Poisson summation formula. We assume here $(n_1, p_1)=1$. In the case when $p_1|n_1$, we write $n_1=p_1^{s}n_1'$, $(n_1',p_1)=1$. We  then work with $n_1'$ instead of $n_1$ and whole analysis goes through.  Splitting the sum over $q$ into dyadic blocks  $q \sim C$ with $C \ll Q$ and writing $q = p^{\ell}_1q' = p^{\ell}_1q'_1q'_2$, $\ell \geq 0$ and $q'_1|(n_1r)^{\infty}, (q'_2,n_1r) =1$. 
		Thus we see that 
		\begin{align*}
			\mathcal{S}_2 \ll\,&N^{\epsilon}\,\mathop{\text{sup}}_{C \ll Q}\, \frac{ N^{5/12}\,}{p^{3/2}_1D_2^{1/4}\,C^{3/2}\,r^{2/3}}\,\sum_{\pm}\,\sum_{\pm}\sum_{\frac{n_1}{(n_1,r)}\ll C}\,n^{1/3}_1\,\sum_{\frac{n_1}{(n_1,r)}|q_1'|(n_1r)^{\infty}}\, \\ \notag
			&\times\, \sum_{n_{2}\ll N_0/n^2_1}\,  \frac{|A_{\pi}(n_{2},n_{1})|}{ n^{1/3}_{2}} \,\left|\displaystyle\sum_{q_2'\sim C/q_1'p_1^\ell}\,\,\sum_{m \ll M_0} \frac{\lambda_f(m)}{m^{1/4}}\,\mathcal{C}^{\pm}_2(n^2_1n_2,m; q) \,\mathfrak{J}^{\pm,\pm}_{2}(n^2_1n_2,m; q)\right|.
		\end{align*} 
		We also split the sum over $m$ into dyadic blocks $m \sim M_1$ with $M_1 \ll M_0$.  On  applying  the Cauchy-Schwarz inequality to the sum over $n_2$, we get 
		\begin{align}\label{MS3}
			\mathcal{S}_2 \ll N^{\epsilon}\mathop{\mathop{\text{sup}}_{C \ll Q}}_{M_1 \ll M_0}\sum_{\pm}\sum_{\pm}\frac{ N^{5/12}\,}{p^{3/2}_1 D_2^{1/4}\,C^{3/2}\,r^{2/3}}\sum_{\frac{n_1}{(n_1,r)}\ll C}\,n^{1/3}_1\sum_{\frac{n_1}{(n_1,r)}|q'_1|(n_1r)^{\infty}}\Theta^{1/2}\,\, \Omega^{1/2}_2,
		\end{align}
		where
		\begin{align}\label{m16}
			\Theta = \sum_{n_{2}\ll N_0/n^2_1} \frac{|A_{\pi}(n_{2},n_{1})|^2}{ n^{2/3}_{2}},
		\end{align}
		\begin{align}\label{O3}
			\Omega_2 = \sum_{n_{2}\ll {N_0}/{n^2_1}}\left|\displaystyle\sum_{q'_2\sim {C}/{p^{\ell}_1q'_1}}\, \sum_{m \sim M_1}\, \frac{\lambda_f(m)}{m^{1/4}} \,\mathcal{C}^{\pm}_2(n^2_1n_2,m; p^{\ell}_1q')\,\mathfrak{J}^{\pm,\pm}_{2}(n^2_1n_2,m; p^{\ell}_1q')\right|^2.    
		\end{align}
		
		The next lemma will prove a bound for $\Theta$. In the proof, we are using the Rankin-Selberg bound on average for the $GL(3)$ Fourier coefficients $A_{\pi}(n_2, n_1)$.
		\begin{lemma}\label{m20}
			For $k\geq 7/6$, we have
			\begin{align*}
				L(n_1,n_2,k) = \sum_{n_1 \ll p_1Cr}\, \frac{(n_1, r )^{1/2} \Theta^{1/2}}{n^k_1} \ll_{ \pi, \epsilon} \, N_0^{1/6+\epsilon}   .
			\end{align*}
		\end{lemma}
		\begin{proof}
			By putting the value of $\Theta$, we get
			\begin{align*}
				L(n_1,n_2,k) =  \sum_{n_1 \ll p_1Cr}\, \frac{(n_1,r)^{1/2}}{n^{k-2/3}_1}\,\frac{1}{n^{2/3}_1}\,  \left(\sum_{n^2_1n_{2}\ll N_0} \frac{|A_{\pi}(n_{2},n_{1})|^2}{n_2^{2/3}} \right)^{1/2}.   
			\end{align*}
			Now we apply the Cauchy-Schwartz inequality in $n_1$-sum, and we get
			\begin{align*}
				L(n_1,n_2,k) \ll \, \left(\sum_{n_1 \ll p_1Cr}\, \frac{(n_1,r)}{n^{2k-4/3}_1}\right)^{1/2}\, \left(\sum_{n^2_1n_{2}\ll N_0}  \frac{|A_{\pi}(n_{2},n_{1})|^2}{(n^2_1n_2)^{2/3}} \right)^{1/2} .    
			\end{align*}
			Note that the first factor is  $O(p_1Cr)^{\epsilon}$.  For the second factor, we apply the Abel summation formula and Rankin-Selberg bound to get 
			\begin{align*}
				\sum_{n^2_1n_{2}\ll N_0}  \frac{|A_{\pi}(n_{2},n_{1})|^2}{(n^2_1n_2)^{2/3}}  \, \ll_{\epsilon} \, N^{1/3+\epsilon}_0.    
			\end{align*}
			This proves our lemma.
		\end{proof}
		
		\begin{lemma}\label{ch1}
			Let $q=p^\ell q'$.  We have 
			\begin{align*}
				\mathcal{C}^{\pm}_2(n^2_1n_2,m; q)=	\mathcal{C}^{\pm}_2(\mathrm{err_1})+ 	\mathcal{C}^{\pm}_2(\mathrm{err_2})+	\mathcal{C}^{\pm}_2(\mathrm{main}),
			\end{align*}
			where  \[ \mathcal{C}^{\pm}_2(\mathrm{err_1})= \delta(\ell=0)  \sum_{d'|q'}\,d'\,\mu\left(\frac{q'}{d'}\right)\,\mathop{\mathop{\sideset{}{^\star}{\displaystyle\sum}_{\alpha \, \mathrm{mod} \, q'r/n_1}}_{\alpha n_1\, \equiv\, m\overline{D_2}\,\,\mathrm{mod} \, d'}}_{  }\,e\left(\frac{\pm \bar{\alpha}n_2 \overline{p_1}}{q'r/n_1}\right),  \] 
			 \[ \mathcal{C}^{\pm}_2(\mathrm{err_2})= \delta(\ell=0)(-p_1\delta(n_2 \equiv 0\, \mathrm{mod}\, p_1)) \sum_{d'|q'}\,d'\,\mu\left(\frac{q'}{d'}\right)\,\mathop{\mathop{\sideset{}{^\star}{\displaystyle\sum}_{\alpha \, \mathrm{mod} \, q'r/n_1}}_{\alpha n_1\, \equiv\, m\overline{D_2}\,\,\mathrm{mod} \, d'}}_{  }\,e\left(\frac{\pm \bar{\alpha}n_2 \overline{p_1}}{q'r/n_1}\right),  \] 
			and for $\ell \geq 0$
			\begin{align*}
				\mathcal{C}^{\pm}_2(\mathrm{main}) =p_1\sum_{d'|q'}\,d'\,\mu\left(\frac{q'}{d'}\right)\ \mathop{\mathop{\sideset{}{^\star}{\displaystyle\sum}_{\alpha \, \mathrm{mod} \, p_1qr/n_1}}_{\alpha n_1\, \equiv\, m\overline{D_2}\,\,\mathrm{mod} \, p_1d'}}e\left(\frac{\pm \bar{\alpha}n_2}{p_1qr/n_1}\right)  \sideset{}{^\star}{\displaystyle\sum}_{a \, \mathrm{mod} \, p_1^{\ell}} e\left(  \frac{(n_1 \alpha -m \overline{ D_2 })a  }{p_1^{1+\ell}}\right),
			\end{align*}
			where for $\ell=0$, the sum over $a$ is  $1$.
		\end{lemma}
		\begin{proof}
			From Lemma \ref{a10}, we have
			\begin{align*}
				\mathcal{C}^{\pm}_2(n^2_1n_2,m; q)= \sideset{}{^\star}{\displaystyle\sum}_{c \, \mathrm{mod} \, p_1^{1+\ell}q'}\,\,\,S\left( r\overline{c}, \pm n_{2}; p_1^{1+\ell} q'r/n_{1}\right)\,e\left(-m  \frac{\overline{c D_2}}{p_1^{1+\ell}q'}\right).
			\end{align*} 
			On applying reciprocity, we write
			$ 	\mathcal{C}^{\pm}_2(n^2_1n_2,m; q)=	\mathcal{C}^{\pm}_3(...)	\mathcal{C}^{\pm}_4(...),$
			where 
			\[ 	\mathcal{C}^{\pm}_3(...)= \sideset{}{^\star}{\displaystyle\sum}_{c \, \mathrm{mod} \, q'}\,\,\,S\left( r\overline{c} \overline{p_1^{1+\ell}} , \pm n_{2} \overline{p_1^{1+\ell}};  q'r/n_{1}\right)\,e\left(-m  \frac{\overline{c D_2} \overline{p_1^{1+\ell}}}{q'}\right),\]
			and
			\[ 	\mathcal{C}^{\pm}_4(...)= \sideset{}{^\star}{\displaystyle\sum}_{c \, \mathrm{mod} \, p_1^{1+\ell}}\,\,\,S\left( n_1\overline{c} \overline{q'} , \pm n_{2}n_1 \overline{q'r };  p_1^{1+\ell} \right)\,e\left(-m  \frac{\overline{c D_2 q'}  }{p_1^{1+\ell}}\right),\]
			Opening the Kloosterman sum and executing the sum $c$ over $q'$, we get 
			\begin{align}\label{c3}
				\mathcal{C}^{\pm}_3(...)= \sum_{d'|q'}\,d'\,\mu\left(\frac{q'}{d'}\right)\,\mathop{\mathop{\sideset{}{^\star}{\displaystyle\sum}_{\alpha \, \mathrm{mod} \, q'r/n_1}}_{\alpha n_1\, \equiv\, m\overline{D_2}\,\,\mathrm{mod} \, d'}}_{  }\,e\left(\frac{\pm \bar{\alpha}n_2 \overline{p_1^{1+\ell}}}{q'r/n_1}\right).
			\end{align}
			Similarly, 	we get 
			\begin{align*}
				\mathcal{C}^{\pm}_4(...)&= \sideset{}{^\star}{\displaystyle\sum}_{c \, \mathrm{mod} \, p_1^{1+\ell}}\,\,\,S\left( n_1\overline{c} \overline{q'} , \pm n_{2}n_1 \overline{q'r };  p_1^{1+\ell} \right)\,e\left(-m  \frac{\overline{c D_2 q'}  }{p_1^{1+\ell}}\right)\\
				&= \sideset{}{^\star}{\displaystyle\sum}_{\alpha \, \mathrm{mod} \, p_1^{1+\ell}}\, e\left( \pm\frac{\bar{\alpha} n_2n_1\overline{q'r}}{p_1^{1+\ell}}\right)      \sideset{}{^\star}{\displaystyle\sum}_{c \, \mathrm{mod} \, p_1^{1+\ell}} e\left(  \frac{n_1\overline{c}\alpha -m \overline{c D_2 }  }{p_1^{1+\ell}}\right).
			\end{align*}
			Consider the sum over $c$.  For $\ell=0$, we get 
			\begin{align*}
				\mathcal{C}^{\pm}_4(...)&=  p_1\, \mathop{\mathop{\sideset{}{^\star}{\displaystyle\sum}_{\alpha \, \mathrm{mod} \, p_1}}_{\alpha n_1\, \equiv\, m\overline{D_2}\,\,\mathrm{mod} \, p_1 }}_{} e\left( \pm\frac{\bar{\alpha} n_2n_1\overline{q'r}}{p_1}\right)-  \sideset{}{^\star}{\displaystyle\sum}_{ \alpha \, \mathrm{mod} \, p_1} e\left( \pm\frac{\bar{\alpha} n_2n_1\overline{q'r}}{p_1^{}}\right) \\
				&=1-p_1\delta(n_2 \equiv 0\, \mathrm{mod}\, p_1)+  p_1\, \mathop{\mathop{\sideset{}{^\star}{\displaystyle\sum}_{\alpha \, \mathrm{mod} \, p_1}}_{\alpha n_1\, \equiv\, m\overline{D_2}\,\,\mathrm{mod} \, p_1 }}_{} e\left( \pm\frac{\bar{\alpha} n_2n_1\overline{q'r}}{p_1}\right). 
			\end{align*}
			For $\ell>0$, 
			by the change of variable $c \rightarrow \bar{c}$ followed by $c \rightarrow a+bp_1^\ell$, we arrive at
			\begin{align*}
				\mathcal{C}^{\pm}_4(...)
				&=p_1 \mathop{\mathop{\sideset{}{^\star}{\displaystyle\sum}_{\alpha \, \mathrm{mod} \, p_1^{1+\ell}}}_{\alpha n_1\, \equiv\, m\overline{D_2}\,\,\mathrm{mod} \, p_1 }}_{}  e\left( \pm\frac{\bar{\alpha} n_2n_1\overline{q'r}}{p_1^{1+\ell}}\right)      \sideset{}{^\star}{\displaystyle\sum}_{a \, \mathrm{mod} \, p_1^{\ell}} e\left(  \frac{n_1a \alpha -ma \overline{ D_2 }  }{p_1^{1+\ell}}\right).
			\end{align*}
			Combining the above expressions of $	\mathcal{C}^{\pm}_4(...)$ with \eqref{c3} we get the lemma. 
		\end{proof}
		\vspace{0.3cm}
		
		Using  Lemma \ref{ch1} and the  inequality $|a+b+c|^2 \ll |a|^2+|b|^2+|c|^2$, $a, b,c \in \mathbb{C}$, we see that
		$$\Omega_2 \ll \Omega_2(\mathrm{main})+\Omega_2(\mathrm{err_i}),$$ where $\Omega_2(\mathrm{main})$ and  $\Omega_2(\mathrm{err_i})$ denote the contribution of  $	\mathcal{C}^{\pm}_2(\mathrm{main})$ and $	\mathcal{C}^{\pm}_2(\mathrm{err_i})$ (i$=1,2$) to $\Omega_2$ respectively. 
		\vspace{0.3cm}
		
		\section{Application of  Poisson summation }
		Opening the absolute valued square and smoothing out the sum over $n_2$ with an appropriate smooth  and compactly supported function $V$ in  $\Omega_2(\mathrm{main})$, we get
		\begin{align}\label{a18}
			\Omega_2(\mathrm{main})\ll\,  \,\displaystyle\sum_{q_2'\sim C/p_1^\ell q'_1}\,\displaystyle\sum_{q''_2\sim C/p_1^{\ell}q'_1}\,\,\sum_{m \sim M_1}\,\sum_{m_1 \sim M_1} \frac{|\lambda_f(m)\,\lambda_f(m_1)|}{(mm_1)^{1/4}}\,|\mathcal{H}_2(\mathrm{main})|,
		\end{align} where $q'' = q_1'q''_2$ and
		\begin{align}\label{O32}
			\mathcal{H}_2(\mathrm{main})= &\sum_{n_{2}\in \mathbb{Z}}\,V\left(\frac{n_2}{N_0/n^2_1}\right)\, \mathcal{C}^{\pm}_2(\mathrm{main})\,\overline{\mathcal{C}^{\pm}_2(\mathrm{main})}\,\notag\\
			& \times \, \mathcal{J}^{\pm,\pm}_{2}(n^2_1n_2,m; p^{\ell}_1q')\,\overline{\mathcal{J}^{\pm,\pm}_{2}(n^2_1n_2,m_1; p^{\ell}_1q'')}.
		\end{align}
		By Lemma \ref{ch1}, we have
		\begin{align*}
			\mathcal{C}^{\pm}_2(\mathrm{main}) =p_1\sum_{d'|q'}\,d'\,\mu\left(\frac{q'}{d'}\right)\ \mathop{\mathop{\sideset{}{^\star}{\displaystyle\sum}_{\alpha \, \mathrm{mod} \, p_1qr/n_1}}_{\alpha n_1\, \equiv\, m\overline{D_2}\,\,\mathrm{mod} \, p_1d'}}e\left(\frac{\pm \bar{\alpha}n_2}{p_1qr/n_1}\right)  \sideset{}{^\star}{\displaystyle\sum}_{a \, \mathrm{mod} \, p_1^{\ell}} e\left(  \frac{(n_1\alpha -m \overline{ D_2 })a  }{p_1^{1+\ell}}\right).
		\end{align*}
		Thus 
		\begin{align} \label{Mm}
			\mathcal{C}^{\pm}_2(\mathrm{main})\,\overline{\mathcal{C}^{\pm}_2(\mathrm{main})}& =\, p_1^2\,\sum_{d'|q'}\sum_{d''|q''} d'\,d''\, \mu\left(\frac{q'}{d'}\right)\mu\left(\frac{q''}{d''}\right) \notag \\ \times\,\, &\mathop{\mathop{\sideset{}{^\star}{\displaystyle\sum}_{\alpha \, \mathrm{mod} \, p_1^{1+\ell}q'r/n_1}}_{\alpha n_1\, \equiv\, m\overline{D_2}\,\,\mathrm{mod} \, p_1d'}}\,\,\, \mathop{\mathop{\sideset{}{^\star}{\displaystyle\sum}_{\alpha' \, \mathrm{mod} \, p_1^{1+\ell}q''r/n_1}}_{\alpha' n_1\, \equiv\, m_1\overline{D_2}\,\,\mathrm{mod} \, p_1d''}} \, \, e\left(\frac{\pm(\overline{\alpha}q''_2 - \overline{\alpha}'q'_2)n_2}{p_1^{1+\ell}q'_1q'_2q''_2r/n_1}\right) \notag \\
			& \times \sideset{}{^\star}{\displaystyle\sum}_{a \, \mathrm{mod} \, p_1^{\ell}}\,\,  \sideset{}{^\star}{\displaystyle\sum}_{a' \, \mathrm{mod} \, p_1^{\ell}} e\left(  \frac{n_1a \alpha -n_1a'\alpha'- ma \overline{ D_2} +m_1a'\overline{ D_2} }{p_1^{1+\ell}}\right).
		\end{align}
		
		Consider the sum over $n_2$:
		\[ \sum_{n_{2} \in \mathbb{Z}} \,  V\left(\frac{n_2}{N_0/n^2_1}\right) e\left(\frac{\pm(\overline{\alpha}q''_2 - \overline{\alpha}'q'_2)n_2}{p_1^{1+\ell}q'_1q'_2q''_2r/n_1}\right)  \mathcal{J}^{\pm,\pm}_{2}(n^2_1n_2,m; p^{\ell}_1q')\,\overline{\mathcal{J}^{\pm,\pm}_{2}(n^2_1n_2,m_1; p^{\ell}_1q'')}.   \]
		Let $P_2 = \frac{p_1^{1+\ell}q_1'q_2'q''_2r}{n_1}$. On applying Poisson summation with modulus $P_2$, we arrive at the following expression:
		\begin{align}\label{a21}
			\Omega_2(\mathrm{main})\ll\,  \,   \frac{ p^2_1\,N_0}{n^2_1\,M^{1/2}_1}  \sum_{n_{2}\in \mathbb{Z}}\displaystyle\sum_{q_2'\sim C/p_1^\ell q'_1}\,\displaystyle\sum_{q''_2\sim C/p_1^{\ell}q'_1}\,\,\sum_{m \sim M_1}\,\sum_{m_1 \sim M_1}  |\mathfrak{C}_2(\mathrm{main})|\,|\mathcal{G}|,
		\end{align} 
		where 
		\begin{align}\label{a20}
			\notag \mathfrak{C}_2(\mathrm{main}) =\,&\ \sum_{d'|q'}\sum_{d''|q''}\,d' d''\,\mu\left(\frac{q'}{d'}\right)\mu\left(\frac{q''}{d''}\right) \sideset{}{^\star}{\displaystyle\sum}_{a \, \mathrm{mod} \, p_1^{\ell}}  \sideset{}{^\star}{\displaystyle\sum}_{a' \, \mathrm{mod} \, p_1^{\ell}} e\left(  \frac{ - ma \overline{ D_2} +m_1a'\overline{ D_2} }{p_1^{1+\ell}}\right)\\
			&\,\times \, \mathop{\mathop{\mathop{\sideset{}{^\star}{\displaystyle\sum}_{\alpha \, \mathrm{mod} \, p^{1+\ell}_1q'_1 q'_2r/n_1}}_{\alpha n_1\, \equiv\, m\overline{D_2}\,\,\mathrm{mod} \, p_1d' }}\,\,\, \mathop{\mathop{\sideset{}{^\star}{\displaystyle\sum}_{\alpha' \, \mathrm{mod} \, p^{1+\ell}_1q'_1q''_2r/n_1}}_{\alpha' n_1\, \equiv\, m_1\overline{D_2}\,\,\mathrm{mod} \, p_1d''}}{ } }_{\overline{\alpha}q''_2 - \overline{\alpha}'q'_2 \equiv -n_2\;\mathrm{mod}\; P_2}\,e\left(  \frac{ n_1a \alpha -n_1a'\alpha' }{p_1^{1+\ell}}\right),
		\end{align}
		and
		\begin{align}\label{a26}
			\mathcal{G} = \int_{\mathbb{R}} V(\xi)\, \mathcal{J}^{\pm,\pm}_{2}(N_0 \xi,m; p^{\ell}_1q')\,\overline{\mathcal{J}^{\pm,\pm}_{2}(N_0 \xi,m_1; p^{\ell}_1q'')}   e\left(-\frac{n_2N_0\xi}{P_2n^2_1}\right)d\xi.
		\end{align} 
		Note that we have used Deligne's bound $\lambda_f(n) \ll n^\epsilon$ (which is known for holomorphic cusp forms). For a  Maass  cusp form $f$
		we can avoid using  $\lambda_f(n) \ll n^\epsilon$. Indeed  the  inequality	$\lambda_{f}(m)\lambda_{f}(m_1)\ll |\lambda_{f}(m)|^2+|\lambda_{f}(m_1)|^2$ followed by the Ramanujan bound on average 
		\[ \sum_{n \ll N}|\lambda_{f}(m)|^2 \ll(t_fP_1P_2)^\epsilon N\] 
		suffices for our purpose.  We analyse $\Omega_2(\mathrm{err_i})$  in a similar way. Recall that 
			\begin{align}\label{a err}
			\Omega_2(\mathrm{err_i})\ll\,  \,\displaystyle\sum_{q_2'\sim C/ q'_1}\,\displaystyle\sum_{q''_2\sim C/q'_1}\,\,\sum_{m \sim M_1}\,\sum_{m_1 \sim M_1} \frac{|\lambda_f(m)\,\lambda_f(m_1)|}{(mm_1)^{1/4}}\,|\mathcal{H}_2(\mathrm{err_i})|,
		\end{align}
		with 
			\begin{align}\label{O32}
			\mathcal{H}_2(\mathrm{err_i})= &\sum_{n_{2}\in \mathbb{Z}}\,V\left(\frac{n_2}{N_0/n^2_1}\right)\, \mathcal{C}^{\pm}_2(\mathrm{err_i})\,\overline{\mathcal{C}^{\pm}_2(\mathrm{err_i})}\,\notag\\
			& \times \, \mathcal{J}^{\pm,\pm}_{2}(n^2_1n_2,m; q')\,\overline{\mathcal{J}^{\pm,\pm}_{2}(n^2_1n_2,m_1; q'')}.
		\end{align}
		Recall that this case occurs only for $\ell=0$. We proceed to apply Poisson summation to $\mathcal{H}_2(\mathrm{err_i})$.  We consider    $\mathcal{H}_2(\mathrm{err_2})$ first.  In this case we have $p_1|n_2$.  By the change of variable $n_2 \rightarrow n_2p_1$, we see that
			\begin{align}
			\mathcal{C}^{\pm}_2(\mathrm{err_2})\,\overline{\mathcal{C}^{\pm}_2(\mathrm{err_2})}& =\, p_1^2\,\sum_{d'|q'}\sum_{d''|q''} d'\,d''\, \mu\left(\frac{q'}{d'}\right)\mu\left(\frac{q''}{d''}\right) \notag \\ \times\,\, &\mathop{\mathop{\sideset{}{^\star}{\displaystyle\sum}_{\alpha \, \mathrm{mod} \, q'r/n_1}}_{\alpha n_1\, \equiv\, m\overline{D_2}\,\,\mathrm{mod} \, d'}}\,\,\, \mathop{\mathop{\sideset{}{^\star}{\displaystyle\sum}_{\alpha' \, \mathrm{mod} \, q''r/n_1}}_{\alpha' n_1\, \equiv\, m_1\overline{D_2}\,\,\mathrm{mod} \, d''}} \, \, e\left(\frac{\pm(\overline{\alpha}q''_2 - \overline{\alpha}'q'_2)n_2}{q'_1q'_2q''_2r/n_1}\right). 
		\end{align}
	On applying Poisson summation with modulus  $ \frac{q_1'q_2'q''_2r}{n_1}$, we arrive at the following expression:
	\begin{align}\label{a21'}
		\Omega_2(\mathrm{err_2})\ll\,  \,   \frac{ p_1\,N_0}{n^2_1\,M^{1/2}_1}  \sum_{n_{2}\in \mathbb{Z}}\displaystyle\sum_{q_2'\sim C/ q'_1}\,\displaystyle\sum_{q''_2\sim C/q'_1}\,\,\sum_{m \sim M_1}\,\sum_{m_1 \sim M_1}  |\mathfrak{C}_2(\mathrm{err_2})|\,|\mathcal{G}|,
	\end{align} 
	where 
	\begin{align}
	\notag \mathfrak{C}_2(\mathrm{err_2}) = \sum_{d'|q'}\sum_{d''|q''}\,d' d''\,\mu\left(\frac{q'}{d'}\right)\mu\left(\frac{q''}{d''}\right)  
		 \mathop{\mathop{\mathop{\sideset{}{^\star}{\displaystyle\sum}_{\alpha \, \mathrm{mod} \, q'_1 q'_2r/n_1}}_{\alpha n_1\, \equiv\, m\overline{D_2}\,\,\mathrm{mod} \, d' }}\, \mathop{\mathop{\sideset{}{^\star}{\displaystyle\sum}_{\alpha' \, \mathrm{mod} \, q'_1q''_2r/n_1}}_{\alpha' n_1\, \equiv\, m_1\overline{D_2}\,\,\mathrm{mod} \, d''}} }_{\overline{\alpha}q''_2 - \overline{\alpha}'q'_2 \equiv -n_2\;\mathrm{mod}\; {q_1'q_2'q''_2r}/{n_1}  }1
	\end{align}
	and
	\begin{align}\label{a26}
\mathcal{G}=	\mathcal{G}(\ell=0) = \int_{\mathbb{R}} V(\xi)\, \mathcal{J}^{\pm,\pm}_{2}(N_0 \xi,m; q')\,\overline{\mathcal{J}^{\pm,\pm}_{2}(N_0 \xi,m_1; q'')}   e\left(-\frac{n_2N_0\xi}{ P_2n_1^2}\right)d\xi.
	\end{align} 
Note that the integral transform in this case is exactly  the same as that in $	\Omega_2(\mathrm{main})$ for $\ell=0$. Moreover  $\mathfrak{C}_2(\mathrm{err_2})$ is also similar to $\mathfrak{C}_2(\mathrm{err_2})$ (at $\ell=0$) without the $p_1$-part in the $\alpha$, $\alpha^\prime$-sum. Similarly on applying Poisson summation to the $n_2$-sum in $\Omega_2(\mathrm{err_1})$, we get 
	\begin{align}\label{a21'}
	\Omega_2(\mathrm{err_1})\ll\,  \,   \frac{ N_0}{n^2_1\,M^{1/2}_1}  \sum_{n_{2}\in \mathbb{Z}}\displaystyle\sum_{q_2'\sim C/ q'_1}\,\displaystyle\sum_{q''_2\sim C/q'_1}\,\,\sum_{m \sim M_1}\,\sum_{m_1 \sim M_1}  |\mathfrak{C}_2(\mathrm{err_1})|\,|\mathcal{G}(\mathrm{err_1})|,
\end{align} 
where 
\begin{align}
	\notag \mathfrak{C}_2(\mathrm{err_1}) = \sum_{d'|q'}\sum_{d''|q''}\,d' d''\,\mu\left(\frac{q'}{d'}\right)\mu\left(\frac{q''}{d''}\right)  
	\mathop{\mathop{\mathop{\sideset{}{^\star}{\displaystyle\sum}_{\alpha \, \mathrm{mod} \, q'_1 q'_2r/n_1}}_{\alpha n_1\, \equiv\, p_1m\overline{D_2}\,\,\mathrm{mod} \, d' }}\, \mathop{\mathop{\sideset{}{^\star}{\displaystyle\sum}_{\alpha' \, \mathrm{mod} \, q'_1q''_2r/n_1}}_{\alpha' n_1\, \equiv\, p_1m_1\overline{D_2}\,\,\mathrm{mod} \, d''}} }_{\overline{\alpha}q''_2 - \overline{\alpha}'q'_2 \equiv -n_2\;\mathrm{mod}\; {q_1'q_2'q''_2r}/{n_1}  }1
\end{align}
and
\begin{align}\label{a26}
	\mathcal{G}(\mathrm{err_1})  = \int_{\mathbb{R}} V(\xi)\, \mathcal{J}^{\pm,\pm}_{2}(N_0 \xi,m; q')\,\overline{\mathcal{J}^{\pm,\pm}_{2}(N_0 \xi,m_1; q'')}   e\left(-\frac{n_2p_1N_0\xi}{ P_2n_1^2}\right)d\xi.
\end{align} 

		\section{Estimates for  Character Sums}
		In this section we will estimate the character sum in the two cases: zero frequency $n_2 =0$ and the non-zero frequency $n_2 \neq 0$. We denote the contribution of the zero and non-zero frequency to   $\mathfrak{C}_{2}(\mathrm{main})$ by $\mathfrak{C}_{2, 0}(\mathrm{main})$ and $\mathfrak{C}_{2, \neq 0}(\mathrm{main})$ respectively.  $\mathfrak{C}_{2, 0}(\mathrm{err_i})$ and $\mathfrak{C}_{2, \neq 0}(\mathrm{err_i})$, i$=1,2$ are defined similarly. 
		\begin{lemma} \label{ch0}
			For $n_2=0$, we have $q'=q_1'q_2'=q_1'q_2''$ and 
			\begin{align*}
				\mathfrak{C}_{2, 0}(\mathrm{main})  \ll\, {p^{3\ell}_1q'_1 q'_2r} \, \mathop{\sum_{d'|q'}\sum_{d''|q'}}_{{p_1}(d', d'')|\,(m-m_1)} \,(d', d'').  
			\end{align*}
			\begin{align*}
	 \mathfrak{C}_2(\mathrm{err_1}), \mathfrak{C}_2(\mathrm{err_2})  \ll \, {q'_1 q'_2r} \, \mathop{\sum_{d'|q'}\sum_{d''|q'}}_{{}(d', d'')|\,(m-m_1)} \,(d', d'').  
		\end{align*}
		\end{lemma}
		\begin{proof}
			From  \eqref{a20}, we have
			\begin{align*}
				\mathfrak{C}_{2, 0}(\mathrm{main}) = \,&\ \sum_{d'|q'}\sum_{d''|q''}\,d' d''\,\mu\left(\frac{q'}{d'}\right)\mu\left(\frac{q''}{d''}\right) \sideset{}{^\star}{\displaystyle\sum}_{a \, \mathrm{mod} \, p_1^{\ell}} \,\,\, \sideset{}{^\star}{\displaystyle\sum}_{a' \, \mathrm{mod} \, p_1^{\ell}} e\left(  \frac{ - ma \overline{ D_2} +m_1a'\overline{ D_2} }{p_1^{1+\ell}}\right)\\
				&\,\times \, \mathop{\mathop{\mathop{\sideset{}{^\star}{\displaystyle\sum}_{\alpha \, \mathrm{mod} \, p^{1+\ell}_1q'_1 q'_2r/n_1}}_{\alpha n_1\, \equiv\, m\overline{D_2}\,\,\mathrm{mod} \, p_1d' }}\,\,\, \mathop{\mathop{\sideset{}{^\star}{\displaystyle\sum}_{\alpha' \, \mathrm{mod} \, p^{1+\ell}_1q'_1q''_2r/n_1}}_{\alpha' n_1\, \equiv\, m_1\overline{D_2}\,\,\mathrm{mod} \, p_1d''}}{ } }_{\overline{\alpha}q''_2 - \overline{\alpha}'q'_2 \equiv \,0\;\mathrm{mod}\; P_2}\,e\left(  \frac{ n_1a \alpha -n_1a'\alpha' }{p_1^{1+\ell}}\right),
			\end{align*}
			Consider the congruence condition
			\begin{align*}
				\overline{\alpha}q''_2 - \overline{\alpha}'q'_2 \equiv 0\;\mathrm{mod}\;P_2=p^{1+\ell}_1q'_1q'_2 q''_2r/n_1.
			\end{align*} On reducing this congruence equation modulo $q'_2$ and $q''_2$, we  get
			\begin{align*}
				q'_2 = q''_2\hspace{0.4cm} \text{and} \,\,\,\,\, \alpha \equiv \alpha' \;\mathrm{mod}\;p^{1+\ell}_1q'_1q'_2 r/n_1.
			\end{align*}
			Thus we see that (with $q'=q_1'q_2'=q_1'q_2''$)
			\begin{align*}
				\mathfrak{C}_{2, 0}(\mathrm{main}) &\ll \,\, p_1^{2\ell}\,\,\sum_{d'|q'}\sum_{d''|q'} d'\,d''\,
				\mathop{\mathop{\mathop{\sideset{}{^\star} \sum_{\substack{\beta\; \mathrm{mod}\;p^{1+\ell}_1 \\ \beta n_1 \equiv m\overline{D_2}\;\mathrm{mod}\;p_1}}}_{\beta n_1 \equiv m_1\overline{D_2}\;\mathrm{mod}\;p_1 }}} \mathop{\mathop{\mathop{\sideset{}{^\star} \sum_{\substack{\alpha\; \mathrm{mod}\;q'_1q'_2r/n_1\\ \alpha n_1 \equiv m\overline{D_2}\;\mathrm{mod}\;d'}}}_{\alpha n_1 \equiv m_1\overline{D_2}\;\mathrm{mod}\;d'' }}}1.
			\end{align*} 
			Note that 
			\begin{align}\label{14}
				\mathop{\mathop{\mathop{\sideset{}{^\star} \sum_{\substack{\beta\; \mathrm{mod}\;p^{1+\ell}_1 \\ \beta n_1 \equiv m\overline{D_2}\;\mathrm{mod}\;p_1}}}_{\beta n_1 \equiv m_1\overline{D_2}\;\mathrm{mod}\;p_1 }}}1 \ll p_1^{\ell} \, \delta(m-m_1 \equiv 0\, \mathrm{mod} \, p_1).
			\end{align}
			 We are left with 
			\begin{align*}
				\sum_{d'|q'}\sum_{d''|q'} d'\,d''\,
				\mathop{\mathop{\mathop{\sideset{}{^\star} \sum_{\substack{\alpha\; \mathrm{mod}\;q'_1q'_2r/n_1\\ \alpha n_1 \equiv m\overline{D_2}\;\mathrm{mod}\;d'}}}_{\alpha n_1 \equiv m_1\overline{D_2}\;\mathrm{mod}\;d'' }}}1,
			\end{align*} 
			which is dominated by 
			\begin{align*}
				{q'_1 q'_2r} \, \mathop{\sum_{d'|q'}\sum_{d''|q'}}_{(d', d'')|\,(m-m_1)} \,(d', d'').  
			\end{align*}
			This combined with \eqref{14} gives the first part of the  lemma. Second part of the lemma follows by similar arguments.   
		\end{proof}
		\vspace{0.3cm}
		
		\begin{lemma}\label{a31}
		We have
			\begin{align*}
				\mathfrak{C}_{2, \neq 0}(\mathrm{main})\,  \ll \,\frac{p^{3\ell}_1q_1^{'2} r (m,n_1)}{n_1} \,\mathop{\mathop{\mathop{\sum \sum}_{d'_{2}|(q'_2, n_1 q''_2D_2 +  mn_2)}}_{d''_{2}|(q''_2, n_1 q'_2D_2 +  m_1n_2)}}_{p_1 | -\overline{m}q''_2 + \overline{m_1}q'_2 + n_2 \overline{n_1D_2}} d'_{2} d''_{2},
			\end{align*}
			\begin{align*}
		\mathfrak{C}_{2, \neq 0}(\mathrm{err_1}) \ll \,\frac{q_1^{'2} r (m,n_1)}{n_1} \,\mathop{\mathop{\mathop{\sum \sum}_{d'_{2}|(q'_2, n_1 q''_2D_2 +  p_1mn_2)}}_{d''_{2}|(q''_2, n_1 q'_2D_2 + p_1 m_1n_2)}}  d'_{2} d''_{2},
		\end{align*}
		\begin{align*}
		 \mathfrak{C}_{2, \neq 0}(\mathrm{err_2})  \ll \,\frac{q_1^{'2} r (m,n_1)}{n_1} \,\mathop{\mathop{\mathop{\sum \sum}_{d'_{2}|(q'_2, n_1 q''_2D_2 +  mn_2)}}_{d''_{2}|(q''_2, n_1 q'_2D_2 +  m_1n_2)}}  d'_{2} d''_{2}.
	\end{align*}
		\end{lemma}
		\begin{proof}
			From  \eqref{a20}, we recall that 
			\begin{align*}
				\notag \mathfrak{C}_{2, \neq 0}(\mathrm{main}) =\,&\ \sum_{d'|q'}\sum_{d''|q''}\,d' d''\,\mu\left(\frac{q'}{d'}\right)\mu\left(\frac{q''}{d''}\right) \sideset{}{^\star}{\displaystyle\sum}_{a \, \mathrm{mod} \, p_1^{\ell}}\,\,\,  \sideset{}{^\star}{\displaystyle\sum}_{a' \, \mathrm{mod} \, p_1^{\ell}} e\left(  \frac{ - ma \overline{ D_2} +m_1a'\overline{ D_2} }{p_1^{1+\ell}}\right)\\
				&\,\times \, \mathop{\mathop{\mathop{\sideset{}{^\star}{\displaystyle\sum}_{\alpha \, \mathrm{mod} \, p^{1+\ell}_1q'_1 q'_2r/n_1}}_{\alpha n_1\, \equiv\, m\overline{D_2}\,\,\mathrm{mod} \, p_1d' }}\,\,\, \mathop{\mathop{\sideset{}{^\star}{\displaystyle\sum}_{\alpha' \, \mathrm{mod} \, p^{1+\ell}_1q'_1q''_2r/n_1}}_{\alpha' n_1\, \equiv\, m_1\overline{D_2}\,\,\mathrm{mod} \, p_1d''}}{ } }_{\overline{\alpha}q''_2 - \overline{\alpha}'q'_2 \equiv -n_2\;\mathrm{mod}\; P_2}\,e\left(  \frac{ n_1a \alpha -n_1a'\alpha' }{p_1^{1+\ell}}\right).
			\end{align*}
			We estimate the sum over $a$, $a'$ trivially to arrive at 
			\begin{align}\label{A61}
				\mathfrak{C}_{2, \neq 0}(\mathrm{main}) \,\ll\, \,p^{2\ell}_1\,\,\mathfrak{C}'_{2, \neq 0}\,\,\mathfrak{C}''_{2, \neq 0},
			\end{align}
			where 
			\begin{align*}
				\mathfrak{C}'_{2, \neq 0} = \sum_{d'_{1}|q'_1}\sum_{d''_{1}|q'_1} d'_{1}\,d''_{1}\,  \mathop{\mathop{\mathop{\sideset{}{^\star}{\displaystyle\sum}_{\alpha \, \mathrm{mod} \, p_1^{1+\ell}q'_1r/n_1}}_{\alpha n_1\, \equiv\, m\overline{D_2}\,\,\mathrm{mod} \, p_1d'_{1}}}\,\, \mathop{\sideset{}{^\star}{\displaystyle\sum}_{\alpha' \, \mathrm{mod} \, p_1^{1+\ell}q'_1r/n_1}}_{\alpha' n_1\, \equiv\, m_1\overline{D_2}\,\,\mathrm{mod} \, p_1d''_{1}}}_{\overline{\alpha}q''_2 - \overline{\alpha}'q'_2 \equiv -n_2\;\mathrm{mod}\; p_1^{1+\ell}q'_1r/n_1}\,1,
			\end{align*} and
			\begin{align*}
				\mathfrak{C}''_{2, \neq 0} = \sum_{d'_{2}|q'_2}\sum_{d''_{2}|q''_2} d'_{ 2}\,d''_{2}\, \mathop{\mathop{\mathop{\sideset{}{^\star}{\displaystyle\sum}_{\alpha \, \mathrm{mod} \, q'_2}}_{\alpha n_1\, \equiv\, m\overline{D_2}\,\,\mathrm{mod} \, d'_{2}}}\,\,\, \mathop{\mathop{\sideset{}{^\star}{\displaystyle\sum}_{\alpha' \, \mathrm{mod} \, q''_2}}_{\alpha' n_1\, \equiv\, m_1\overline{D_2}\,\,\mathrm{mod} \, d''_{2}}} }_{ \overline{\alpha}q''_2 - \overline{\alpha}'q'_2 \equiv -n_2\;\mathrm{mod}\;  q'_2q''_2}\,1.
			\end{align*}
			In $\mathfrak{C}'_{2, \neq 0}$, we have
			\begin{align*}
				&\overline{\alpha}q''_2 - \overline{\alpha}'q'_2+n_2 \equiv 0\; \mathrm{mod}\;p_1^{1+\ell}q'_1r/n_1\\	
				\Longrightarrow\,\,\,\,\,&{\alpha}'(q''_2+n_2 \alpha)\equiv \alpha q'_2\; \mathrm{mod}\;p_1^{1+\ell}q'_1r/n_1\\
				\Longrightarrow\,\,\,\,\,& {\alpha}' \equiv \alpha q'_2 \overline{(q''_2+n_2 \alpha)}\; \mathrm{mod}\;p_1^{1+\ell}q'_1r/n_1.
			\end{align*}
			Thus   $\alpha'$ is determined in terms of $\alpha$ and we arrive at 
			\begin{align}\label{A62}
				\mathfrak{C}'_{2, \neq 0} &\ll \,\sum_{d'_{1}|q'_1}\sum_{d''_{1}|q'_1} d'_{1}\,d''_{1}\, \mathop{\mathop{\sideset{}{^\star} {\sum}_{\substack{\alpha\; \mathrm{mod}\; p_1^{1+\ell}q'_1r/n_1\\ \alpha \,n_1\equiv m\overline{D_2}\;\mathrm{mod}\;p_1d'_{1} } }}}  \delta( -\overline{m}q''_2 + \overline{m_1}q'_2 + n_2 \overline{n_1D_2}\equiv 0 \, \mathrm{mod}\,p_1 ) \notag\\
					& \ll \, \frac{p_1^\ell q_1^{'2} r(m, n_1)}{n_1} \delta( -\overline{m}q''_2 + \overline{m_1}q'_2 + n_2 \overline{n_1D_2}\equiv 0 \, \mathrm{mod}\,p_1 ) .
			\end{align}
		Next we consider	 $\mathfrak{C}''_{2, \neq 0}$.  
	  On using the congruence conditions and $(n_1, p_1q'_2q''_2) =1$, we get
			\begin{align*}
				&n_1 \alpha \equiv m\overline{D_2}\;\mathrm{mod}\;d'_{2}, \hspace{1cm} n_1 \alpha' \equiv m_1\overline{D_2}\;\mathrm{mod}\;d''_{ 2}\\
				\Longrightarrow \,\,\,\,\,& \alpha \equiv \overline{n_1} m\overline{D_2}\;\mathrm{mod}\;d'_{2} ,\hspace{1cm} \alpha' \equiv \overline{n_1} m_1\overline{D_2}\;\mathrm{mod}\;d''_{2}.
			\end{align*} 
			Thus $\alpha$ and $\alpha'$ are determined  modulo $d'_{2}$ and $d''_{2}$ respectively. On using the congruence condition modulo $q'_2q''_2$, we get
			\begin{align}\label{A,}
				\mathfrak{C}''_{2, \neq 0} \ll \mathop{\mathop{\mathop{\sum \sum}_{d'_{2}|(q'_2, n_1 q''_2D_2 +  mn_2)}}_{d''_{ 2}|(q''_2, n_1 q'_2D_2 +  m_1n_2)}}d'_{2}\, d''_{2}.
			\end{align}
			Thus  combining  \eqref{A61}, \eqref{A62} and \eqref{A,} we get the first part of the  lemma. Estimates for $ \mathfrak{C}_{2, \neq 0}(\mathrm{err_i})$ follow similarly. 
		\end{proof}
		\vspace{0.3cm}
		
		\section{Analysis of the Integral transform}
		In this section, we will divide the analysis of the integral transform $\mathcal{G}$ defined in \eqref{a26} into the two cases, namely zero frequency $n_2 =0$ and non-zero frequency $n_2 \neq 0$. We denote the contribution of this integral transform by $\mathcal{G}_{0}$ and $\mathcal{G}_{\neq 0}$ in the said cases of zero and non-zero frequency. 
		
		\begin{lemma}\label{a24}
			Let $\lambda = {(NN_0)^{1/3}}/{p_1 C r^{1/3}}$, we have
			\begin{align*}
				\mathcal{G} \ll \, {\max}\left\{{t},\, \lambda \right\}^{-1}.   
			\end{align*}
			Also, the integral transform $\mathcal{G}_{\neq 0}$ is negligibly small unless $$n_2 \ll \frac{p_1C^2\lambda rn_1}{p^{\ell}_1q'_1N_0}N^{\epsilon} := N_2.$$  \end{lemma}
		\begin{proof}
			From equation \eqref{a13}, we have
			\begin{align*}
				\mathcal{J}^{\pm,\pm}_{2, y}(N_0\xi,m,q) =  \int_0^ \infty U(y) \,y^{-it} \, e\left(\pm\,\frac{2\sqrt{mNy}}{p_1q \sqrt{D_2}} \pm \frac{3(NN_0 \xi (y+w))^{1/3}}{p_1qr^{1/3}} \right) dy.   
			\end{align*}
			As we know that $|w| \ll 1/K$, we can write
			\begin{align}\label{m100}
				(y_1+w)^{1/3} = y^{1/3}_1\,\left(1 + \frac{1}{3} \frac{w}{y} - \frac{1}{9} \left(\frac{w}{y}\right)^2 + ... \right) = y^{1/3}_1 + \text{lower order terms}.
			\end{align}
			We will do the analysis for the leading term (first term) and the rest are lower-order terms, that can be analyzed in a similar way.  By changing the variable $y \longrightarrow y^2$, we can rewrite the above integral transform as
			\begin{align*}
				\mathcal{J}^{\pm,\pm}_{2, y}(N_0\xi,m,q) =  \int_0^ \infty S(y)\,e\left(P(y)\right),   
			\end{align*}
			where $S(y) = 2y\,U(y^2 - w)$ is the new smooth and compactly supported function and the phase function is given by
			\begin{align*}
				P(y) = -\frac{t}{2\pi} \log(y^2) \pm \frac{2\sqrt{mN}}{p_1q \sqrt{D_2}}y \pm  3 \lambda\,\xi^{1/3} y^{2/3}.   
			\end{align*}
			By differentiating twice with respect to $y$, we get
			\begin{align*}
				P''(y) =  \frac{t}{\pi y^2} \mp  \frac{2\lambda \xi^{1/3}}{3\,y^{4/3}} \gg {\max}\left\{{t},\, \lambda \right\}. 
			\end{align*}
			By using the second derivative bound (Lemma \ref{sdb}), we get
			\begin{align*}
				\mathcal{J}^{\pm,\pm}_{2, y}(N_0\xi,m,q) \ll \max \left\{{t},\, \lambda \right\}^{-1/2}.    
			\end{align*}
			Now, using the above bound for the integral transform $ \mathcal{J}^{\pm,\pm}_{2, y}(N_0\xi,m,q)$ and treating the $\xi$-integral trivially, we get our desired bound for $ \mathcal{G}$. But the above bound holds only when $t \not\asymp \lambda$. In the other situation, we get our desired bound in $L^2$-sense by opening the absolute valued square in equation \eqref{a26} and then applying integration by parts to $\xi$-integral and getting restriction $|y_1 - y_2| \ll 1/t \asymp 1/\lambda$.\\
			
			Now we prove the second part of our lemma. we have
			\begin{align*}
				\mathcal{J}^{\pm,\pm}_{2, y}(N_0\xi,m,q)\,\,\overline{\mathcal{J}^{\pm,\pm}_{2, y}(N_0\xi,m_1,q')} = &\, \,\int_0^ \infty\,\int_0^ \infty\,U(y_1) U(y_2)\,\, e\left(\frac{-t}{2\pi}(\log y_1-\log y_2)\right) \\
				&\times\, e\left(\pm\,3\lambda \left({y^{1/3}_1} -{y^{1/3}_1}\right)\,\xi^{1/3} \right)\\
				&\times e\left(\pm\,\frac{2\sqrt{N}}{p_1C\sqrt{D_2}} \left({\sqrt{my_1}} - {\sqrt{m_1y_2}}\right) \right) dy_1\,dy_2.   
			\end{align*}
			Notice that
			\begin{align*}
				\frac{\partial^j}{\partial \xi^j}\,(\mathcal{J}^{\pm,\pm}_{2, y}(N_0\xi,m,q)\,\,\overline{\mathcal{J}^{\pm,\pm}_{2, y}(N_0\xi,m_1,q')} ) \ll \, \lambda^j. 
			\end{align*} 
			Now, on applying integration by parts $j$-times to $\xi$-integral in \eqref{a26}, we arrived at the following expression
			\begin{align*}
				\mathcal{G}_{\neq 0} \ll \,\lambda^j \times \left(\frac{p_1C^2r}{p^{\ell}_1q'_1n_1}. \frac{n^2_1}{|n_2|N_0}\right)^j. 
			\end{align*}
			The integral $\mathcal{G}_{\neq 0} $ will be negligibly small unless
			\begin{align*}
				|n_2| \ll \, \frac{p_1C^2 \lambda\, r\,n_1}{p^{\ell}_1q'_1N_0}N^{\epsilon}  := \, N_2.  
			\end{align*}
			This completes the proof of our lemma. 
		\end{proof}
		\vspace{0.3cm}
		
		\begin{lemma}\label{in0}
			For the zero frequency case with $m= m_1$, we have
			\begin{align*}
				\mathcal{G}_{0}(...)\ll\, \frac{1}{t}.     
			\end{align*}
			For $m \neq m_1$ and $\lambda = \left({NN_0}/{p^3_1 q^3 r}\right)^{1/3}$, we have
			\begin{align*}
				\mathcal{G}_{0} \ll \, \frac{1}{\sqrt{t\lambda}} \hspace{0.5cm} \text{with the condition,} \hspace{0.5cm} m-m_1 \ll \frac{M_1}{\lambda}.     
			\end{align*}
		\end{lemma}
		\begin{proof}
			For $n_2 = 0$, we have $q = q'$. In this situation and also using the observation of equation \eqref{m100}, our integral transform given in equation \eqref{a26} will reduce to,
			\begin{align*}
				\mathcal{G}_{0} =&\, \,\int_0^ \infty\,\int_0^ \infty\,U(y_1) U(y_2)\, \int_{\mathbb{R}}\, V(\xi)\, e\left(\pm\,3\,\lambda \left({y^{1/3}_1} -{y^{1/3}_2}\right)\xi^{1/3} \right) \\
				&\times e\left(\frac{-t}{2\pi}(\log y_1-\log y_2) \pm \frac{2\sqrt{N}}{p_1q\sqrt{D_2}} \left({\sqrt{my_1}} - {\sqrt{m_1y_2}}\right) \right) d\xi\,dy_1\,dy_2.
			\end{align*}
			For the case $m = m_1$, from Lemma \ref{a24}, we have 
			\begin{align*}
				\mathcal{G}_{0} \ll\, \frac{1}{t}.  
			\end{align*}
			For the other case when $m \neq m_1$, by changing the variable $\xi \longrightarrow \xi^3$ and applying integration by parts to $\xi$-integral, we get that it will be negligibly small unless 
			\begin{align*}
				y_2 - y_1  \ll \frac{1}{\lambda}, \hspace{0.5cm} \text{where} \hspace{0.5cm} \lambda = \left(\frac{NN_0}{p^3_1 q^3 r}\right)^{1/3}.
			\end{align*}
			Changing the variable $y_2 - y_1 = y_3$ with $y_3 \ll 1/\lambda$, then $y_1$-integral (say $\mathcal{I}_{y_1}$) becomes
			\begin{align*}
				\mathcal{I}_{y_1} = \int_0^ \infty\, U(y_1)\,U(y_1 + y_3)\,e\left(t\log\left(1+\frac{y_3}{y_1}\right) \pm \frac{2\sqrt{N}}{p_1q\sqrt{D_2}} \left({\sqrt{my_1}} - {\sqrt{m_1(y_1+y_3)}}\right) \right) dy_1 .    
			\end{align*}
			We notice that the integral $\mathcal{I}_{y_1}$ will be negligibly small unless
			\begin{align*}
				\frac{t}{\lambda} \asymp  \frac{\sqrt{N} (\sqrt{m} - \sqrt{m_1})}{p_1q\sqrt{D_2}} \hspace{0.5cm} \Longrightarrow \hspace{0.5cm} m-m_1 \ll\, \frac{M_1}{\lambda}.
			\end{align*}
			Since we have $y_3 \ll 1/\lambda$. By using the similar reasoning as in equation \eqref{m100} and changing the variable $y_1 \longrightarrow y^2_1$, we can rewrite $y_1$-integral as
			\begin{align*}
				\mathcal{I}_{y_1}  = \int_0^ \infty\, 4 y^2_1\,U(y^2_1)\,U(y^2_1 + y_3)\,e\left(t\left(\frac{y_3}{y_1}\right) \pm \frac{2\sqrt{N}}{p_1q\sqrt{D_2}} \left({\sqrt{m}} - {\sqrt{m_1}}\right) y_1 \right) dy_1 . \end{align*}
			Now, by applying the second derivative bound to $y_1$-integral, we get
			\begin{align*}
				\mathcal{I}_{y_1} \ll \, \frac{\sqrt{\lambda}}{\sqrt{t}}.    
			\end{align*}
			Now, combining the estimates for $y_1$ and $y_3$ integral, we finally get our final bound in the case when $m \neq m_1$, i.e.
			\begin{align*}
				\mathcal{G}_{0} \ll \, \frac{1}{\sqrt{t\lambda}} \hspace{0.5cm} \text{with the condition,} \hspace{0.5cm} m-m_1 \ll \frac{M_1}{\lambda}.  
			\end{align*} 
			This proves our lemma.
		\end{proof}
		\vspace{0.3cm}
		
		\noindent In the case of non-zero frequency $n_2 \neq 0$, our integral transform $\mathcal{G}_{\neq 0}$ become
		\begin{align*}
			&\,  \int_{\mathbb{R}}\, V(\xi)\, e\left(\pm\,3\left(\frac{NN_0}{p^3_1 r}\right)^{1/3} \left(\frac{y^{1/3}_1}{q_1q_2} -\frac{y^{1/3}_2}{q_1q'_2}\right)\xi^{1/3} - \frac{n_2N_0\xi}{Pn^2_1} \right) \,d\xi\\
			& \times \int_0^ \infty\,\int_0^ \infty\,U(y_1) U(y_2)\,e\left(\frac{-t}{2\pi}(\log y_1-\log y_2) \pm \frac{2\sqrt{N}}{p_1\sqrt{D_2}} \left(\frac{\sqrt{my_1}}{q_1q_2} - \frac{\sqrt{m_1 y_2}}{q_1q'_2}\right) \right)dy_1\,dy_2. 
		\end{align*}
		As we can see that the $\xi$-integral is oscillating like $K$ and both $y_1, y_2$ integrals are oscillating like $\text{max}\left\{t, \sqrt{NM_1}/p_1C\sqrt{D_2} \right\}$. We aim to get the square root savings in all three integrals.
		The treatment of the integral transform $\mathcal{G}_{\neq 0}$  will be similar to  that in \cite{r26}.

		\begin{lemma}\label{d1}
			Let $X, Z$ be real numbers with $|Z| \ll 1$, define
			\begin{align*}
				T(X, 3\lambda Z) = \int_{\mathbb{R}} \xi^2\, V(\xi^3)\, e\left(3\lambda Z\xi - X \xi^3\right)\,d\xi.
			\end{align*} 
			\begin{enumerate}
				\item We have $T(X, 3\lambda Z) =  O(N^{-A})$ if $X \gg \lambda^{1+\epsilon}$.
				\item For $X \gg N^{\epsilon}$, we have $T(X, 3\lambda Z) =  O(N^{-A})$ unless $(4/9)^{1/3} -1/6 < \lambda Z/X < 9^{1/3}+1/6$. For $1/3 < \lambda Z/X < 3$, we let $\mathfrak{s} = 2X(\lambda Z/X)^{3/2},$ then
				\begin{align*}
					T(X, 3\lambda Z)  = |X|^{-1/2} \, e(\mathfrak{s})\, T_1(\mathfrak{s}) + O_A(N^{-A}),
				\end{align*} for any $A>0$, with $T_1(\mathfrak{s}) \in C^{\infty}_c (X/3, 6X)$ satisfying $\mathfrak{s}^{j}\,T^{j}_1(\mathfrak{s}) \ll 1. $
			\end{enumerate}
		\end{lemma}
		\begin{proof}
			See \cite[Lemma $5.1$]{r26}.
		\end{proof}
		
		\begin{lemma}\label{d6}
			Let $X = {n_2N_0}/{Pn^2_1}$, $\lambda = \left({NN_0}/{p^3_1C^3r}\right)^{1/3}$, $Y = {\max}\left\{t, \sqrt{NM_1}/p_1C\sqrt{D_2} \right\}$ and
			\begin{align}\label{k4}
				\mathcal{G}_{\neq 0}(X) = \int_{\mathbb{R}} V(\xi)\, \mathcal{J}^{\pm,\pm}_1(N_0\xi,m,q)\,\, \overline{\mathcal{J}^{\pm,\pm}_1(N_0\xi ,m_1,q')}\, e\left(-{X\xi}\right)\, d\xi.
			\end{align}
			\begin{enumerate}
				\item We have $\mathcal{G}_{\neq 0}(X) = O(N^{-A})$ if $X \gg \lambda^{1+\epsilon}$.
				\item For $X \ll \lambda^{3+\epsilon}\,Y^{-2}$, we have $\mathcal{G}_{\neq 0}(X) \ll Y^{-1}$.
				\item For $X$ such that $\lambda^{3+\epsilon}\,Y^{-2} \ll X \ll \lambda^{2+\epsilon}\,Y^{-1}$, we have $\mathcal{G}_{\neq 0}(X) \ll Y^{-1}\, |X|^{-1/3}$.
				\item For $X \gg \lambda^{2+\epsilon} \,Y^{-1}$, we have
				$\mathcal{G}_{\neq 0}(X) \ll Y^{-1}\, |X|^{-1/2}$.
			\end{enumerate}
		\end{lemma}
		
		\begin{proof}
			See \cite[Lemma 4.3]{r26}. 	
		\end{proof}
		\vspace{0.3cm}

		\section{Estimates for the zero Frequency}
		Let $\Omega_{2, 0}(\mathrm{main})$ and $\mathcal{S}_{2, 0}(\mathrm{main})$ denotes the contribution of the  zero frequency to $\Omega_{2}(\mathrm{main})$  and  $\mathcal{S}_{2}(\mathrm{main})$.  $\Omega_{2, 0}(\mathrm{err_i})$ and $\mathcal{S}_{2, 0}(\mathrm{err_i})$ are defined similarly.  We have the following lemma.
		\vspace{0.3cm}
		
		\begin{lemma}\label{a43}
			We have
			\begin{align*}
				\mathcal{S}_{2, 0}=	\mathcal{S}_{2, 0}(\mathrm{main})+ \mathcal{S}_{2, 0}(\mathrm{err_1})+\mathcal{S}_{2, 0}(\mathrm{err_2}) \,\ll \,\,r^{1/2}\,p^{3/4}_1\,K^{3/4}\,N^{3/4}N^{\epsilon}.
			\end{align*}
		\end{lemma}
		\begin{proof}
			We consider $	\Omega_{2, 0}(\mathrm{main})$. 
			For $n_2 =0$, we have $q_2' = q_2''$. Recall from   \eqref{a21} that 
			\begin{align}\label{a41}
				\Omega_{2, 0}(\mathrm{main})\, \ll\, \frac{ p^2_1\,N_0}{n^2_1\,M^{1/2}_1} \,\displaystyle\sum_{q'_2\sim C/p^{\ell}_1q'_1}\,\,\sum_{m \sim M_1}\,\,\sum_{m_1 \sim M_1}\,\,|\mathfrak{C}_{2, 0}(\mathrm{main})|\,|\mathcal{G}_{0}|,
			\end{align}
			\vspace{0.2cm}
			Using Lemma \ref{ch0} for the character sum $\mathfrak{C}_{2, 0}(\mathrm{main})$ and Lemma \ref{in0} for the  integral transform $\mathcal{G}_{0}$ in the two different cases $m =m_1$ and $m \neq m_1$ in the above expression, we get
			\begin{align*}
				\Omega_{2, 0}(\mathrm{main}) \ll&\,  \frac{ p^2_1\,N_0}{n^2_1\,M^{1/2}_1}\,\frac{p^{3\ell}_1 q'_1r}{t}\,\displaystyle\sum_{q'_2\sim C/p^{\ell}_1q'_1}\,q'_2\,\sum_{m_1 \sim M_1} \, \,\sum_{d'|q'}\sum_{d''|q'}\, (d',d'')\\
				&+\,\frac{ p^2_1\,N_0}{n^2_1\,M^{1/2}_1}\, \frac{p^{3\ell}_1 q'_1r}{\,\sqrt{t\lambda}}\,\,\displaystyle\sum_{q'_2\sim C/p^{\ell}_1q'_1}\,q'_2\,\mathop{\mathop{\sum_{m_1 \sim M_1}\,\,\sum_{m_1 \sim M_1}}_{p_1(d', d'')|(m-m_1)}}_{m - m_1\ll M_1/\lambda}\,\,\sum_{d'|q'}\sum_{d''|q'}\, (d', d'')\,\\
				\ll&\,\, \frac{ p^2_1\,N_0\,\,p^{3\ell}_1\,q'_1r}{\,n^2_1\,M^{1/2}_1}\,\left(\frac{C}{p^{\ell}_1q'_1}\right)^2\, \left(\frac{M_1\,C}{p^{\ell}_1\,t} + \frac{M^2_1}{p_1 \lambda \sqrt{t \lambda }}\right).
			\end{align*} 
			Plugging the value $\lambda =  \left({NN_0}/{p^3_1 q^3 r}\right)^{1/3} \asymp K$ and observing that the first term of the bracket is dominating, we finally get
			\begin{align*}
				\Omega_{2, 0}(\mathrm{main}) \ll\,  \frac{r\,p^2_1\,N_0\,M^{1/2}_1\,C^3}{n^2_1\,q'_1\,t}\,.
			\end{align*} 
		For  $	\Omega_{2, 0}(\mathrm{err_i}) $ we follow the same steps and get 
			\begin{align*}
			\Omega_{2, 0}(\mathrm{err_i}) \ll\,  \frac{r\,p^2_1\,N_0\,M^{1/2}_1\,C^3}{n^2_1\,q'_1\,t}\,.
		\end{align*} 
	Thus \begin{align}\label{k}
			\Omega_{2, 0}:=  	\Omega_{2, 0}(\mathrm{main}) +		\Omega_{2, 0}(\mathrm{err_1})+ \Omega_{2, 0}(\mathrm{err_2}) \ll \frac{r\,p^2_1\,N_0\,M^{1/2}_1\,C^3}{n^2_1\,q'_1\,t} 
	\end{align}
			From  \eqref{MS3}, we have
			\begin{align*}
				\mathcal{S}_{2,0} \,\ll&\,N^{\epsilon}\mathop{\mathop{\text{sup}}_{C \ll Q}}_{M_1 \ll M_0} \frac{ N^{5/12}\,}{p^{3/2}_1 D_2^{1/4}\,C^{3/2}\,r^{2/3}}\,\sum_{\pm}\sum_{\pm}\sum_{\frac{n_1}{(n_1,r)}\ll C}\,n^{1/3}_1\sum_{\frac{n_1}{(n_1,r)}|q'_1|(n_1r)^{\infty}}\, \Theta^{1/2}\,\Omega_{2,0}^{1/2}
			\end{align*}
		On	plugging the bound  \eqref{k}, we see that  $\mathcal{S}_{2, 0}$ is dominated by
			\begin{align*}
				& N^{\epsilon}\mathop{\mathop{\text{sup}}_{C \ll Q}}_{M_1 \ll M_0} \frac{ N^{5/12}\,}{p^{3/2}_1D_2^{1/4}\,C^{3/2}\,r^{2/3}}\,\sum_{\frac{n_1}{(n_1,r)}\ll C}\,n^{1/3}_1\,\,\sum_{\frac{n_1}{(n_1,r)}|q'_1|(n_1r)^{\infty}} \,\, \left( \frac{ r\,p^2_1\,N_0\,M^{1/2}_1\,C^3\,}{n^2_1\,q'_1\,t}\right)^{1/2}\,\Theta^{1/2} \\
				\ll&\,N^{\epsilon}\mathop{\mathop{\text{sup}}_{C \ll Q}}_{M_1 \ll M_0} \frac{r^{1/2}\,N^{5/12}\,N^{1/2}_0\,M^{1/4}_1}{\,p^{1/2}_1\,D_2^{1/4}\,t^{1/2}\,r^{2/3}}\,\sum_{\frac{n_1}{(n_1,r)}\ll C}\,\,\,\sum_{\frac{n_1}{(n_1,r)}|q'_1|(n_1r)^{\infty}} \,\, \frac{1}{n^{2/3}_1\,(q'_1)^{1/2}}\,\Theta^{1/2}\\
				\ll&\,N^{\epsilon}\mathop{\mathop{\text{sup}}_{C \ll Q}}_{M_1 \ll M_0} \,\frac{r^{1/2}\,N^{5/12}\,N^{1/2}_0\,M^{1/4}_1}{p^{1/2}_1\,D_2^{1/4}\,t^{1/2}\,r^{2/3}}\,\,\sum_{\frac{n_1}{(n_1,r)}\ll C}\,\frac{(n_1,r)^{1/2}}{n^{7/6}_1}\,\Theta^{1/2}.
			\end{align*} 
			Using Lemma \ref{m20} to $n_1$-sum, we get
			\begin{align*}
				\mathcal{S}_{2, 0} \, \ll\, \frac{r^{1/2}\,N^{5/12}\,N^{2/3}_0\,M^{1/4}_0}{p^{1/2}_1\,D_2^{1/4}\,t^{1/2}\,r^{2/3}}N^{\epsilon}.
			\end{align*}
			Now on plugging the estimates
			\begin{align*}
				N_0 \ll \frac{(p_1QK)^3r}{N}N^{\epsilon}, \hspace{0.3cm}  \hspace{0.3cm} M_0 \ll \frac{(p_1Qt)^{2}D_2}{N }N^{\epsilon}, 
			\end{align*} 
			we get
			\begin{align*}
				\mathcal{S}_{2, 0} \, \ll \,r^{1/2}\,p^{3/4}_1\,K^{3/4}\,N^{3/4}\,N^{\epsilon}.
			\end{align*}
			This proves the lemma.
		\end{proof}
		\vspace{0.2cm}
		
		\section{Estimates for the non-zero frequencies} Let $\Omega_{2, \neq 0}$ and   $\mathcal{S}_{2, \neq 0}$  denote the contribution of the non-zero frequencies $n_2\neq 0$ to  $\Omega_2$   and $\mathcal{S}_{2}$. Let 
		\[ \Omega_{2, \neq 0}:=\Omega_{2, \neq 0}(\mathrm{main})+\Omega_{2, \neq 0}(\mathrm{err_1})+\Omega_{2, \neq 0}(\mathrm{err_2}),\]
	where 	$\Omega_{2, \neq 0}(\mathrm{main})$  denote the contribution of 	$\mathfrak{C}_{2, \neq 0}(\mathrm{main})$  to $\Omega_{2, \neq 0}$. 	$\Omega_{2, \neq 0}(\mathrm{err_i})$  are defined similarly. 
		   We divide the estimation of  non-zero frequency in the two cases corresponding to the small and generic size of parameter $q$. We write
		\begin{align*}
			\mathcal{S}_{2, \neq 0} \,\ll\, S_{small}(N) + S_{gen}(N),
		\end{align*}
		where $S_{small}(N)$ and $S_{gen}(N)$ denotes the contribution for small $q$ case $(q \sim C \ll QK/t)$ and generic $q$ case $(q \sim C \gg QK/t)$, respectively. Firstly, we will find the bound for $S_{small}(N)$.  We have the following lemma.

		\begin{lemma}\label{a44}
			For the small modulus case $q \sim C \ll QK/t$, we have
			\begin{align*}
				S_{small}(N) \ll \,\,\,\frac{r^{1/2}\,p^{1/4}_1\,p^{1/2}_2\,K^{5/4}\,N^{3/4}}{ \,t^{1/2}}.
			\end{align*}
		\end{lemma}
		\begin{proof}
			Recall from \eqref{a21}, we have	
			\begin{align*}
				\Omega_{2, \neq 0}(\mathrm{main}) \ll  \frac{ p^2_1\,N_0}{n^2_1\,M^{1/2}_1} \,\displaystyle\sum_{q'_2\sim C/p^{\ell}_1q'_1}\,\displaystyle\sum_{q''_2\sim C/p^{\ell}_1q'_1}\,\sum_{n_2 \in \mathbb{Z}}\,\,\sum_{m \sim M_1}\,\,\sum_{m_1 \sim M_1}\,|\mathfrak{C}_{2, \neq 0}(\mathrm{main})|\,\,|\mathcal{G}_{\neq 0}|.
			\end{align*}
			Plugging the bound of the character sum $\mathfrak{C}_{2, \neq 0}(\mathrm{main})$ given in Lemma \ref{a31} and the integral transform $\mathcal{G}_{\neq 0}$ given in Lemma \ref{a24}, we get 
			\begin{align*}
				\Omega_{2, \neq 0}(\mathrm{main}) \,\ll &\,\,\frac{ p^2_1\,N_0}{n^2_1\,M^{1/2}_1}\, \frac{p^{3\ell}_1(q'_1)^2\,r}{n_1\,\text{max}\left\{t, \lambda \right\}} \,\displaystyle\sum_{q'_2\sim C/p^{\ell}_1q'_1}\,\,\displaystyle\sum_{q''_2\sim C/p^{\ell}_1q'_1}\,\, \sum_{d'_2|q'_2}\, \sum_{d''_2|q''_2} \,d'_2\, d''_2 \\
				&\times\,\mathop{\mathop{\mathop{\sum_{m \sim M_1}\,\sum_{m_1 \sim M_1}}_{n_1 q''_2 +  mn_2\overline{D_2} \equiv 0\;\mathrm{mod}\;d'_2}}_{n_1 q'_2 +  m_1n_2\overline{D_2} \equiv 0\;\mathrm{mod}\;d''_2}}_{\overline{m}q''_2 + \overline{m_1}q'_2 + n_2 \overline{n_1D_2} \equiv 0 \;\mathrm{mod}\;p_1} (m,n_1)\,\sum_{0<|n_{2}|<N_2}\,1.\\ 
			\end{align*} For $q \sim C \ll QK/t$, we have $\text{max}\left\{t, \lambda \right\} = \lambda$. By changing the variable $q'_2 \longrightarrow q'_2d'_2$ and $q''_2 \longrightarrow q''_2d''_2$ , we get
			\begin{align}\label{m2}
				\Omega_{2, \neq 0}(\mathrm{main}) \ll&\,\,  \frac{p^2_1\,p^{3\ell}_1\,N_0\,(q'_1)^2\,r}{n^3_1\,M^{1/2}_1\,\lambda} \,\sum_{d'_2\leq C/p^{\ell}_1q'_1}\, \sum_{d''_2 \leq C/p^{\ell}_1q'_1} d'_2\,d''_2 \,\displaystyle\sum_{q'_2\sim C/d'_2p^{\ell}_1q'_1}\,\displaystyle\sum_{q''_2\sim C/d''_2p^{\ell}_1q'_1}\notag\\
				&\times \mathop{\mathop{\mathop{\sum_{m \sim M_1}\,\sum_{m_1 \sim M_1}}_{n_1 q''_2d''_2 +  mn_2\overline{D_2} \equiv 0 \;\mathrm{mod}\;d'_2}}_{n_1 q'_2d'_2 +  m_1n_2\overline{D_2} \equiv 0 \;\mathrm{mod}\;d''_2}}_{\overline{m}q''_2d''_2 + \overline{m_1}q'_2d'_2 + n_2 \overline{n_1D_2} \equiv 0 \;\mathrm{mod}\;p_1}\, (m,n_1)\,\sum_{0<|n_{2}|<N_2}\,1.
			\end{align} 
			Now, we count the number of $m$ or $m_1$ and $d'$ or $d''$ with the help of congruence conditions. By multiplying $m\,m_1$ into the congruence modulo $p_1$, we get
			\begin{align*}
				m\,(q'_2d'_2 + m_1n_2\overline{n_1D_2} ) \equiv m_1q''_2d''_2 \;\mathrm{mod}\;p_1 .
			\end{align*}
			We have divided the analysis for counting and finding the bound of $\Omega_{2, \neq 0}(\mathrm{main})$ into the following four cases. For the first two cases, we fix $q'_2, n_2, d'_2, n_1$, and for the next two cases, we fix $q''_2, n_2, d''_2, n_1$.\\

			\textbf{Case 1}: If $(p_1, q'_2d'_2 + m_1n_2\overline{n_1D_2} ) = 1$ and $n_1 q'_2d'_2 +  m_1n_2\overline{D_2} \equiv 0 \;\mathrm{mod}\;d''_2$ but the integer $n_1 q'_2d'_2 +  m_1n_2\overline{D_2}$ is not identically $0$. By using the first assumption, we count $m$ as
			
			\begin{align*}
				\mathop{\mathop{\sum_{m \sim M_1}}_{n_1 q''_2d''_2 +  mn_2\overline{D_2} \equiv 0 \;\mathrm{mod}\;d'_2}}_{m\, \equiv \overline{(q'_2d'_2 + m_1n_2\overline{n_1D_2} )}\,m_1q''_2d''_2 \;\mathrm{mod}\;p_1 }\,1 \ll \, (d'_2, n_2) \left(1 + \frac{M_1}{p_1d'_2}\right).
			\end{align*} 
			Also, we notice from the second assumption that $d''_2$ will be a factor $n_1 q'_2d'_2 +  m_1n_2\overline{D_2}$  So, $d''_2$-sum is bounded by  $O(N^\epsilon)$.  By using these two estimates into above expression of $\Omega_{\neq 0} $, we see that $	\Omega_{2, \neq 0}(\mathrm{main}) $ is bounded by 
				\begin{align*}
			&\,\, \frac{p^2_1\,p^{3\ell}_1\,N_0\,C\,q'_1\,r}{p^{\ell}_1\,n^3_1\,M^{1/2}_1\,\lambda} \,\sum_{d'_2\leq C/p^{\ell}_1q'_1}\, d'_2\,\displaystyle\sum_{q'_2\sim C/d'_2p^{\ell}_1q'_1}\,
				\sum_{m_1 \sim M_1}(m_1,n_1)\,\,\sum_{0<|n_{2}|<N_2}\, (d'_2, n_2) \left(1 + \frac{M_1}{p_1d'_2}\right).\\
			\end{align*}
			Finally, executing the remaining sums trivially, we get
			\begin{align*}
				\Omega_{2, \neq 0}(\mathrm{main})\, \ll&\,\, \frac{p^2_1\,p^{3\ell}_1\,N_0\,C\,q'_1\,r}{p^{\ell}_1\,n^3_1\,M^{1/2}_1\,\lambda} \,\frac{C}{p^{\ell}_1 q'_1}\,M_1\,N_2\, \left(\frac{M_1}{p_1} + \frac{C}{p^{\ell}_1q'_1}\right).   \end{align*}
			Plugging the value of $N_2 = N^{\epsilon}{p_1C^2\lambda rn_1}/{p^{\ell}_1q'_1N_0}$ from Lemma \ref{a24} and noticing that $(M_1/p_1 + C/p^{\ell}_1q'_1) \ll M_1/p_1$, we finally get
			\begin{align*}
				\Omega_{2, \neq 0}(\mathrm{main})\, \ll \, \frac{r^{2}\,p^2_1\,C^4\,M^{3/2}_1}{\,n^2_1\,q'_1\,}.
			\end{align*}\\
			
			\textbf{Case 2:} If $(p_1, q'_2d'_2 + m_1n_2\overline{n_1D_2} ) = 1$ and $n_1 q'_2d'_2 +  m_1n_2\overline{D_2} = 0$, identically. In this case, we count the number of $m$ as in the above case, but from the second condition, we notice that the number of $n_1, q'_2, d'_2$ are bounded by $O(N^\epsilon)$. By putting these estimates into the expression of $\Omega_{2, \neq 0}(\mathrm{main})$ given in \eqref{m2}, we will get the same bound for $\Omega_{2, \neq 0}(\mathrm{main})$ as in the above case.\\
			
			\textbf{Case 3:} If $(p_1, q'_2d'_2 + m_1n_2\overline{n_1D_2} ) \neq 1$ and $n_1 q''_2d''_2 +  mn_2\overline{D_2} \equiv 0 \;\mathrm{mod}\;d'_2$ but the integer $n_1 q''_2d''_2 +  mn_2\overline{D_2}$ is not identically $0$. So, by using the first assumption, we count $m_1$ as
			
			\begin{align*}
				\mathop{\mathop{\sum_{m _1\sim M_1}}_{n_1 q'_2d'_2 +  m_1n_2\overline{D_2} \equiv 0 \;\mathrm{mod}\;d''_2}}_{m_1q''_2d''_2  \equiv 0 \;\mathrm{mod}\;p_1}\,1 \ll \, (d''_2, n_2) \left(1 + \frac{M_1}{p_1d''_2}\right).   
			\end{align*}
			Also, we notice from the second assumption that $d'_2$ will be a factor $n_1 q''_2d''_2 +  mn_2\overline{D_2}$. So, $d'_2$-sum is bounded by  $O(N^\epsilon)$. By using these two estimates into above expression of $\Omega_{2, \neq 0}(\mathrm{main})$, we get that $\Omega_{2, \neq 0}(\mathrm{main})$ is dominated by 
			
			\begin{align*}
				\frac{p^2_1\,p^{2\ell}_1\,N_0\,C\,q'_1\,r}{p^{\ell}_1\,n^3_1\,M^{1/2}_1\,\lambda} \,\sum_{d''_2\leq C/p^{\ell}_1q'_1}\, d''_2\,\displaystyle\sum_{q''_2\sim C/d''_2p^{\ell}_1q'_1}\,
				\sum_{m \sim M_1}(m,n_1)\,\sum_{0<|n_{2}|<N_2}\, (d''_2, n_2) \left(1 + \frac{M_1}{p_1d''_2}\right).
			\end{align*}
			Finally, executing the remaining sums trivially, we get the same bound as in case $1$ for $\Omega_{2, \neq 0}(\mathrm{main})$.\\
			
			\textbf{Case 4:} If $(p_1, q'_2d'_2 + m_1n_2\overline{n_1D_2}) \neq 1$ and $n_1 q''_2d''_2 +  mn_2\overline{D_2} = 0$, identically. We will again get the same bound for $\Omega_{2, \neq 0}(\mathrm{main})$ as in Case $1$ by proceeding as in the above cases. Analysis for $\Omega_{2, \neq 0}(\mathrm{err_i})$ is exactly the same. Thus we get
			\[ \Omega_{2, \neq 0} \ll \frac{r^{2}\,p^2_1\,C^4\,M^{3/2}_1}{\,n^2_1\,q'_1\,}. \]
			\vspace{0.3cm}

			\noindent Now recall from \eqref{MS3}, we have
			\begin{align*}
				S_{small}(N) \ll&\,N^{\epsilon}\mathop{\mathop{\text{sup}}_{C \ll QK/t}}_{M_1 \ll M_0} \frac{ N^{5/12}}{p^{3/2}_1\,D_2^{1/4}\,C^{3/2}\,r^{2/3}}\sum_{\frac{n_1}{(n_1,r)}\ll p_1C}n^{1/3}_1 \sum_{\frac{n_1}{(n_1,r)}|q'_1|(n_1r)^{\infty}} \Theta^{1/2} \left(\Omega_{2, \neq 0}\right)^{1/2}.
			\end{align*}
			Plugging the above  bound of $\Omega_{2, \neq 0}$, we get that $S_{small}(N)$ is dominated by
			\begin{align*}
				&\,N^{\epsilon}\mathop{\mathop{\text{sup}}_{C \ll QK/t}}_{M_1 \ll M_0} \,\,\frac{ N^{5/12}}{p^{3/2}_1\,D_2^{1/4}\,C^{3/2}\,r^{2/3}}\sum_{\frac{n_1}{(n_1,r)}\ll C}n^{1/3}_1\sum_{\frac{n_1}{(n_1,r)}|q'_1|(n_1r)^{\infty}} \left(\frac{r^{2}\,p^2_1\,C^4\,M^{3/2}_1}{ \,n^2_1\,q'_1\,}\right)^{1/2}\Theta^{1/2}\\
				&\ll\,N^{\epsilon}\mathop{\mathop{\text{sup}}_{C \ll QK/t}} \,\,\frac{ r^{1/3}\,N^{5/12}\,\,C^{1/2}\,M^{3/4}_0}{\,p^{1/2}_1\,D_2^{1/4}\,}\,\sum_{\frac{n_1}{(n_1,r)}\ll C}\,\sum_{\frac{n_1}{(n_1,r)}|q'_1|(n_1r)^{\infty}}\,\frac{\Theta^{1/2}}{n^{2/3}_1\,\sqrt{q'_1}}\\
				&\ll\,N^{\epsilon}\mathop{\mathop{\text{sup}}_{C \ll QK/t}} \,\,\frac{ r^{1/3}\,N^{5/12}\,\,C^{1/2}\,M^{3/4}_0}{\,p^{1/2}_1\,D_2^{1/4}\,}\,\sum_{{n_1}\ll Cr}\,\frac{(n_1,r)^{1/2}\,\Theta^{1/2}}{n^{7/6}_1}.
			\end{align*} 
			By using Lemma \ref{m20} for the estimate of $n_1$-sum, we get
			\begin{align*}
				S_{small}(N)\,\ll&\,\, \frac{ r^{1/3}\,N^{5/12}\,\,C^{1/2}\,N^{1/6}_0\,M^{3/4}_0}{\,p^{1/2}_1\,D_2^{1/4}\,}N^{\epsilon} . 
			\end{align*} 
			For the case $q \sim C \ll QK/t$, by plugging the values of $N_0 \ll rN^{1/2}(p_1K)^{3/2} $ and $M_0 \ll p_1D_2K$, we finally get
			
			\begin{align*}
				S_{small}(N) \,\ll\, \,\frac{r^{1/2}\,p^{1/4}_1\,D^{1/2}_2\,K^{5/4}\,N^{3/4}}{\,\,t^{1/2}}N^{\epsilon}.
			\end{align*}
			From the proof of Lemma \ref{a9}, we have
			\begin{align*}
				D_2 = \frac{p_1p_2}{p_1 (q, p_2)} =\,
				\begin{cases}
					p_2 &\text{if}\,\,(p_2, q) = 1, \\
					1 &\text{if}\,\, p_2 | q.
				\end{cases}
			\end{align*}
			Note that, in the case when $(p_2, q) =1$, we have $D_2 = p_2$, and with this we get our desired bound. In the other case when $p_2 |q$, we have $D_2 =1$. So, in this case, we are saving $p^{1/2}_2$ as compared to the co-prime case. Hence we get the lemma. 
		\end{proof}
		\vspace{0.3cm}
		
		\begin{lemma}\label{a45}
			For the generic case $q \sim C \gg QK/t$, we have
			\begin{align*}
				S_{gen}(N)\ll& \, {{r^{1/2}\,p^{1/4}_1\,p^{1/2}_2\,N^{3/4}}}\, \left( K^{3/4} + \frac{t^{2/3}}{K^{1/12}} + \frac{t}{K^{1/2}}\right).
			\end{align*}
		\end{lemma}
		\begin{proof}
			Recall from \eqref{a21}, we have	
			
			\begin{align*}
				\Omega_{2, \neq 0}(\mathrm{main}) \ll  \frac{p^2_1\,N_0}{n^2_1M^{1/2}_1}\sum_{n_{2}\in \mathbb{Z}} \,\displaystyle\sum_{q'_2\sim C/p^{\ell}_1q'_1}\,\displaystyle\sum_{q''_2\sim C/p^{\ell}_1q'_1}\,\,\sum_{m \sim M_1}\,\sum_{m_1 \sim M_1}\,|\mathfrak{C}_{2, \neq 0}(main)|\,|\mathcal{G}_{\neq 0}|.
			\end{align*}
			Note that from lemma \ref{d6}, depending upon the size of variable $X = {p^{\ell}_1q'_1\,n_2 N_0}/{p_1 C^2 rn_1}$, we have different bound for the integral transform $\mathcal{G}_{\neq 0}(X)$. So, we need to evaluate the contribution to $\Omega_{2, \neq 0}(main)$ accordingly. Let $Y = \text{max}\left\{t, {\sqrt{NM_0}}/{p_1C\sqrt{p_2}}\right\}$, we have the following cases:
			\begin{align*}
				&(I)\,\,\,\,\,\,\, X \ll \lambda^3\,Y^{-2} \Longrightarrow\,\, n_2 \ll \frac{\lambda^3\,p_1\,C^2\,rn_1}{p^{\ell}_1q'_1\,N_0 \,Y^2} : = N_3. \\
				&(II)\,\,\,\,\,\,\, \lambda^3\, Y^{-2} \ll X \ll \lambda^2\, Y^{-1} \,\,\,\,\,\,\,\Longrightarrow \,\, N_3 \ll n_2 \ll N_4 := \frac{\lambda^2\,p_1\,C^2\,rn_1}{p^{\ell}_1 q'_1\,N_0 \,Y}.\\
				&(III)\,\,\,\,\,\,\, X \gg\lambda^2\, Y^{-1}\,\Longrightarrow\,\, N_4 \ll n_2 \ll N_2, \hspace{0.5cm} \text{where} \hspace{0.5cm} N_2 = \frac{p_1C^2Krn_1}{p^{\ell}_1q'_1N_0}.
			\end{align*}
			Let $\Omega^I_{\neq 0}, \Omega^{II}_{\neq 0}$ and $\Omega^{III}_{\neq 0}$ denotes the contribution to $\Omega_{2, \neq 0}(main)$ in cases $(I), (II)$ and $(III)$, respectively. Then we have
			
			\begin{align}\label{mmm1}
				\Omega_{2, \neq 0}(\mathrm{main}) \ll \Omega^I_{\neq 0} + \Omega^{II}_{\neq 0}+ \Omega^{III}_{\neq 0}.
			\end{align} 
			For the first case $(I)$, using the bound of character sum $\mathfrak{C}_{2, \neq 0}(main)$ given in Lemma \ref{a31} and the integral transform $\mathcal{G}_{\neq 0}(X) \ll Y^{-1}$ from Lemma \ref{d6}, we get that $\Omega^I_{\neq 0} $ is dominated by
			
			\begin{align*}
				\frac{p^2_1\,N_0\,p^{3\ell}_1(q'_1)^2\,r}{n^3_1\,M^{1/2}_1\,Y}\, \,&\displaystyle\sum_{q'_2\sim C/p^{\ell}_1q'_1}\,\displaystyle\sum_{q''_2\sim C/p^{\ell}_1q'_1}\, \sum_{d'_2|q'_2}\, \sum_{d''_2|q''_2} d'_2\,d''_2 \,\,\sum_{0<|n_{2}|\ll N_3}\,\mathop{\mathop{\mathop{\sum_{m \sim M_1}\,\sum_{m_1 \sim M_1}}_{n_1 q''_2 +  mn_2\overline{D_2} \equiv 0\;\mathrm{mod}\;d'_2}}_{n_1 q'_2 +  m_1n_2\overline{D_2} \equiv 0\;\mathrm{mod}\;d''_2}}_{\overline{m}q''_2 + \overline{m_1}q'_2 + n_2 \overline{n_1D_2} \equiv 0 \;\mathrm{mod}\;p_1}\,\,1.
			\end{align*}
			By changing the variable $q'_2 \longrightarrow q'_2d'_2$, $q''_2 \longrightarrow q''_2d''_2$, we get
			\begin{align}\label{m10}
				\Omega^I_{\neq 0} \ll&\,\,  \frac{p^2_1\,N_0\,p^{3\ell}_1(q'_1)^2\,r}{n^3_1\,M^{1/2}_1\,Y} \,\sum_{d'_2\leq C/p^{\ell}_1q'_1}\, \sum_{d''_2 \leq C/p^{\ell}_1q'_1} d'_2\,d''_2 \,\displaystyle\sum_{q'_2\sim C/p^{\ell}_1d'_2q'_1}\,\displaystyle\sum_{q''_2\sim C/p^{\ell}_1d''_2q'_1}\notag\\
				&\times \,\sum_{0<|n_{2}|<N_3}\, \mathop{\mathop{\mathop{\sum_{m \sim M_1}\,\sum_{m_1 \sim M_1}}_{n_1 q''_2d''_2 +  mn_2\overline{D_2} \equiv 0 \;\mathrm{mod}\;d'_2}}_{n_1 q'_2d'_2 +  m_1n_2\overline{D_2} \equiv 0 \;\mathrm{mod}\;d''_2}}_{{m_1}q''_2d''_2 + {m}(q'_2d'_2 + m_1n_2 \overline{n_1D_2}) \equiv 0 \;\mathrm{mod}\;p_1}\,\,1.
			\end{align} 
			Now, as in the case of a small modulus, we proceed similarly to count $d'_2$ or $d''_2$ and $m$ or $m_1$. We take the same four cases by fixing the parameters $n_1, n_2, d'_2(\text{or}\, d''_2), q'_2(\text{or}\, q''_2)$. In all four cases, we get the same bound for $\Omega^I_{\neq 0}$, which is given by
			\begin{align*}
				\Omega^I_{\neq 0}\ll\, \frac{p^2_1\,N_0\,p^{3\ell}_1(q'_1)^2\,r}{n^3_1\,M^{1/2}_1\,Y}\,\, M_1\,\frac{\lambda^3\,p_1\,C^2\,rn_1}{p^{\ell}_1q'_1\,N_0\,Y^2 }\,\frac{C^2}{(p^{\ell}_1q'_1)^2}\,\left(\frac{M_1}{p_1} + \frac{C}{p^{\ell}_1q'_1}\right) .
			\end{align*}
			On simplification, we get
			\begin{align}\label{k3}
				\Omega^I_{\neq 0} \ll \, \frac{r^2\,p^2_1\,C^4\,M^{3/2}_1\,\lambda^3}{\,n^2_1\,q'_1\,Y^3}.
			\end{align}

			Now we consider the case $(II)$ and find a bound for $\Omega^{II}_{\neq 0}$. By plugging the bound of character sum $\mathfrak{C}_{2, \neq 0}(main)$ given in Lemma \ref{a31} and for the integral transform $\mathcal{G}_{\neq 0}(X) \ll Y^{-1}\,X^{-1/3} = Y^{-1}\left({N_3\,Y^2}/{\lambda^3\,n_2}\right)^{1/3}$ from Lemma \ref{d6}, we get 
			\begin{align*}
				\Omega^{II}_{\neq 0}\, \ll&\,\,  \frac{p^2_1\,N_0\,p^{3\ell}_1(q'_1)^2\,r}{n^3_1\,M^{1/2}_1\,Y} \, \frac{N^{1/3}_3\,Y^{2/3}}{\lambda}\,\sum_{d'_2\leq C/p^{\ell}_1q'_1}\, \sum_{d''_2 \leq C/p^{\ell}_1q'_1} d'_2\,d''_2 \,\displaystyle\sum_{q'_2\sim C/p^{\ell}_1d'_2q'_1}\,\displaystyle\sum_{q''_2\sim C/p^{\ell}_1d''_2q'_1}\notag\\
				&\times \,\sum_{N_3<|n_{2}|<N_4}\,\frac{1}{n^{1/3}_2}\, \mathop{\mathop{\mathop{\sum_{m\sim M_1}\,\sum_{m_1 \sim M_1}}_{n_1 q''_2d''_2 +  mn_2\overline{D_2} \equiv 0 \;\mathrm{mod}\;d'_2}}_{n_1 q'_2d'_2 +  m_1n_2\overline{D_2} \equiv 0 \;\mathrm{mod}\;d''_2}}_{{m_1}q''_2d''_2 + {m}(q'_2d'_2 + m_1n_2 \overline{n_1D_2}) \equiv 0 \;\mathrm{mod}\;p_1}\,\,1.    
			\end{align*}
			Again, for the counting purpose, by taking all four possible cases as in Lemma \ref{a44}, and then trivially estimating the remaining sums, we get 
			\begin{align*}
				\Omega^{II}_{\neq 0}\, \ll&\,\,  \frac{p^2_1\,N_0\,p^{3\ell}_1(q'_1)^2\,r\,N^{1/3}_3}{n^3_1\,M^{1/2}_1\,Y^{1/3}\,\lambda}\,\, M_1\,N^{2/3}_4\,\frac{C^2}{(p^{\ell}_1q'_1)^2}\,\left(\frac{M_1}{p_1} + \frac{C}{p^{\ell}_1q'_1}\right). 
			\end{align*} 
			By plugging the values of $N_3$ and $N_4$, we get
			\begin{align}\label{k2}
				\Omega^{II}_{\neq 0} \ll\,\frac{r^2\,p^2_1\,C^4\,M^{3/2}_1\,\lambda^{4/3}}{ \,q'_1\,n^2_1\,Y^{5/3}}.
			\end{align}
			Let us consider the final case $(III)$ to find bound of $\Omega^{III}_{\neq 0}$. By using the bound of character sum $\mathfrak{C}_{\neq 0}(...)$ from Lemma \ref{a31} and the bound for integral transform $\mathcal{G}_{\neq 0}(X) \ll Y^{-1}\,X^{-1/2} = Y^{-1}\,\left({N_3\,Y^2}/{\lambda^3\,n_2}\right)^{1/2}$ from Lemma \ref{d6} and following the same steps as in case $(II)$, we get
			
			\begin{align*}
				\Omega^{III}_{\neq 0} \ll& \, \frac{rp^2_1\,N_0\,q_1^{'2}p_1^{3\ell}\,N^{1/2}_3}{n^3_1\,M^{1/2}_1\,\lambda^{3/2}}\,\, M_1\,N^{1/2}_2\,\frac{C^2}{(p^{\ell}_1q'_1)^2}\,\left(\frac{M_1}{p_1} + \frac{C}{p^{\ell}_1q'_1}\right). 
			\end{align*}
			Plugging the values of $N_2$ and $N_3$, we get
			
			\begin{align}\label{k1}
				\Omega^{III}_{\neq 0} \ll\, \frac{r^{2}\,p^{2}_1\,C^{4}\,K^{1/2}\,M^{3/2}_1}{\,n^2_1\,q'_1\,Y}.
			\end{align}
			Now, using equations \eqref{k3},\eqref{k2}, \eqref{k1} and \eqref{mmm1}, we get
			
			\begin{align*}
				\Omega_{2, \neq 0}(\mathrm{main}) \ll\, \frac{r^{2}\,p^{2}_1\,C^4\,M^{3/2}_1}{\,n^2_1\,q'_1\,Y} \left( \frac{\lambda^3}{Y^2}+ \frac{\lambda^{4/3}}{Y^{2/3}} + {K^{1/2}}\right).
			\end{align*}
		Analysis for $	\Omega_{2, \neq 0}(\mathrm{err_i}) $ is similar. Thus we get 
			\begin{align*}
			\Omega_{2, \neq 0} \ll 	\Omega_{2, \neq 0}(\mathrm{main})+ 	\Omega_{2, \neq 0}(\mathrm{err_1})+ 	\Omega_{2, \neq 0}(\mathrm{err_2}) \ll\, \frac{r^{2}\,p^{2}_1\,C^4\,M^{3/2}_1}{\,n^2_1\,q'_1\,Y} \left( \frac{\lambda^3}{Y^2}+ \frac{\lambda^{4/3}}{Y^{2/3}} + {K^{1/2}}\right).
		\end{align*}
			\noindent Now recall from \eqref{MS3}, we have
			\begin{align*}
				S_{gen}(N) \ll&\,N^{\epsilon}\mathop{\mathop{\text{sup}}_{QK/t \ll C \ll Q}}_{M_1 \ll M_0} \frac{ N^{5/12}}{p^{3/2}_1\,D_2^{1/4}\,C^{3/2}\,r^{2/3}}\sum_{\frac{n_1}{(n_1,r)}\ll C}n^{1/3}_1 \sum_{\frac{n_1}{(n_1,r)}|q'_1|(n_1r)^{\infty}} \Theta^{1/2} \left(\Omega_{2, \neq 0} \right)^{1/2}.
			\end{align*}
			Plugging the bound of $\Omega_{2, \neq 0}$ and using Lemma \ref{m20} for the estimate of $n_1$-sum, we get 
			
			\begin{align*}
				S_{gen}(N) \ll&\, \,{r^{1/2}\,p^{1/4}_1\,D^{1/2}_2\,N^{3/4}{K^{3/4}}}N^{\epsilon} + \frac{r^{1/2}\,p^{1/4}_1\,D^{1/2}_2\,N^{3/4}\,t^{2/3}}{\,K^{1/12}}N^{\epsilon}+ \frac{r^{1/2}\,p^{1/4}_1\,D^{1/2}_2\,N^{3/4}\,t}{ \,K^{1/2}}N^{\epsilon}\\
				=& \,{{r^{1/2}\,p^{1/4}_1\,D^{1/2}_2\,N^{3/4}}} \,\left( K^{3/4} + \frac{t^{2/3}}{K^{1/12}} + \frac{t}{K^{1/2}}\right)N^{\epsilon}.
			\end{align*}
			From the proof of Lemma \ref{a9}, we have
			\begin{align*}
				D_2 = \frac{p_1p_2}{p_1 (q, p_2)} =\,
				\begin{cases}
					p_2 &\text{if}\,\,(p_2, q) = 1, \\
					1 &\text{if}\,\, p_2 | q.
				\end{cases}
			\end{align*}
			As mentioned in Lemma \ref{a44}, the case $(p_2, q) = 1$ gives the dominating bound. This proves our lemma. 
		\end{proof}
		\vspace{0.2cm}
		
		As we have 
		\begin{align*}
			\mathcal{S}_2 \,  \,\ll\, \mathcal{S}_{2, 0} + \mathcal{S}_{2, \neq 0} \,\ll\, \mathcal{S}_{2, 0} + \mathcal{S}_{small}(N) +  \mathcal{S}_{gen}(N).     
		\end{align*}
		Hence, by using Lemma \ref{a43}, Lemma \ref{a44} and Lemma \ref{a45}, we get
		\begin{align}\label{x}
			\mathcal{S}_2 \ll\,r^{1/2}\,p^{3/4}_1\,K^{3/4}\,N^{3/4}\,N^{\epsilon} +  {{r^{1/2}\,p^{1/4}_1\,p^{1/2}_2\,N^{3/4}}} \left(\frac{K^{5/4}}{t^{1/2}} + K^{3/4} + \frac{t^{2/3}}{K^{1/12}} + \frac{t}{K^{1/2}}\right).    
		\end{align}
		\vspace{0.3cm}
		
		\section{Analysis for $ \mathcal{S}_1$} \label{S1}
		Recall that $ \mathcal{S}_1$ denote the contribution of $b \equiv 0\, \mathrm{mod}\, p_1$, $(p_1, q)=1$  and it is given by 
		\begin{align}\label{hb1}
			\notag \mathcal{S}_1 =& \,\frac{1}{KQp_1}\displaystyle\int_{\mathbb{R}} W(u)\,\int_{\mathbb{R}}\, V_1\left(\frac{\nu}{K}\right)\, \displaystyle\sum_{1\leq q\leq Q}\frac{\psi(q,u)}{q}\sideset{}{^\star}{\displaystyle\sum}_{a \, \mathrm{mod} \, q}\\
			&\times \, \sum_{n=1}^{\infty}A_{\pi}(r,n)e\left(\frac{an}{q }\right) \,v_1(n)\,\sum_{m=1}^{\infty}\lambda_f (m)\,e\left(-\frac{am}{q }\right)\,v_2(m)\, \mathrm{d}u\,\mathrm{d}\nu.
		\end{align}
		We proceed to apply Voronoi formulae as in the previous cases. Analysis for integral transforms will be exacly same in the case. We highlight the changes in the arithmetic part.	We apply  $GL(3)$-Voronoi to the $n$-sum as before and we arrive at
		\begin{align*}
			\frac{N^{2/3+i\nu}}{q} \sum_{\pm}  \sum_{n_{2}\sim (p_1QK)^3/N}  \frac{A_{\pi}(n_{2},1)}{ n^{1/3}_{2}} S\left( \overline{a}, \pm n_{2}; q\right) \eth_1^{\pm}(...),
		\end{align*} where $ \eth_1^{\pm}(...)$ is the same  integral transform as  in Lemma \ref{a8} and we have  square-root cancellations in $ \eth_1^{\pm}(...)$. Thus we save, on applying Weil's bound for  Kloosterman sums,   $\frac{N}{(QK)^{3/2}}$  in this step. Next we apply  the $GL(2)$ Voronoi summation formula to the sum over $m$. Note that the arithmetic conductor $c$, say, is $p_1p_2 q^2$ if $(p_2,q)=1$ and $c=p_1q^2$, if $p_2|q$. Thus we benefit more if  $p_2|q$ in an  application of  $GL(2)$ Voronoi. We proceed with the case $(p_2,q)=1$ and we arrive at
		\begin{align*}
			\frac{ N^{3/4-i(t+\nu)}}{q^{1/2}\,(p_1p_2)^{1/4}}  \sum_{\pm} \sum_{m \sim (Qt)^2p_1p_2/N} \frac{\lambda_f(m)}{m^{1/4}}  e\left(-m  \frac{\overline{a p_1p_2}}{q}\right) \mathcal{I}^{\pm} (...),
		\end{align*}  where $\mathcal{I}^{\pm} (...)$ is as  defined in Lemma \ref{a9}. We save $\frac{N}{(p_1p_2)^{1/2}Qt}$ in this step. We arrive at the following expression 
		\begin{align*}
			\frac{ N^{17/12}}{KQ^{7/2}\,p_1^{5/4}p_2^{1/4}}\,\displaystyle\sum_{q\sim Q}\,\sum_{n_{2}\sim (QK)^{3}/N}  \frac{A_{\pi}(n_{2},1)}{ n^{1/3}_{2}}\,\sum_{m \sim (Qt)^{2}p_1p_2/N} \frac{\lambda_f(m)}{m^{1/4}}\, \mathcal{C}(...)\,\mathcal{J}(...),
		\end{align*}
		where the character sum $\mathcal{C}(...)$ is given by
		\begin{align*}
			\sideset{}{^\star}{\displaystyle\sum}_{a \, \mathrm{mod} \, q}\,\,S\left( \overline{a}, \pm n_{2};  q\right)\,e\left(-m  \frac{\overline{a p_1p_2}}{q}\right)  \rightsquigarrow \, q\, e\left(\pm \frac{\overline{m}p_1p_2n_2}{q}\right).
		\end{align*}  Thus we  save $\sqrt{Q}$ in the $a$-sum.  In the $\nu$-integral, we get the square-root cancellation of size $\sqrt{K}$ by stationary phase analysis.  The total savings so far is given by 
		\begin{align}\label{total sav}
			\frac{N}{(QK)^{3/2}} \times \frac{N}{(p_1p_2)^{1/2}Qt} \times \sqrt{Q} \times p_1 \times \sqrt{K}   = N \times \frac{p_1^{3/2}}{p_2^{1/2}\,t},
		\end{align}
		where the factor $p_1$ in the middle represents the  saving coming from   the $b$-sum, which is just one element in this case. We apply  Cauchy-Schwartz inequality to the  sum over $n_2$ to get rid of $A_{\pi}(1,n)$ and we arrive at
		\begin{align*}
			\,\Bigg(\sum_{n_{2}\sim (QK)^3/N} \,\left|\displaystyle\sum_{q\sim Q}\,\,\sum_{m \sim (Qt)^2p_1p_2/N} \frac{\lambda_f(m)}{m^{1/4}}\,e\left(\frac{\bar{m}p_1p_2n_2}{q}\right)\, \mathcal{J}(...)\right|^2 \Bigg)^{1/2}.
		\end{align*}
		We open the absolute valued square and apply the Poisson summation formula to sum over $n_2$. This step is exactly similar to  previous cases (in fact easier as the  modulus is $qq^\prime$). In the case of zero frequency $n_2 =0$, we will save the whole diagonal length i.e. $Q^3 t^2p_1p_2/N$.  In the case of non-zero frequency  $n_2 \neq 0$, we will save $	{(QK)^{3}}/{ (N\sqrt{K} )}$,  which is satisfactory when
		Optimising the diagonal and the off-diagonal  savings, we see that the  total  savings  are given by 
		\begin{align}
			N \times \frac{p_1^{3/2}}{p_2^{1/2}\,t} \times \left(\frac{Q^3K^3}{N\sqrt{K}}\right)^{1/2}=   N \times \frac{p_1^{3/2}}{p_2^{1/2}\,t} \times \frac{N^{1/4} \sqrt{K}}{p_1^{3/4}}=N \times ( {t^{3/20} p_2^{3/40} p_1^{53/40} }).
		\end{align}
		Thus we see that we have a saving in all ranges of $p_1$, $p_2$ and $t$ in this case and we get 
		 \begin{align}\label{z}
		 	 \mathcal{S}_1 \ll \frac{N}{{t^{3/20} p_2^{3/40} p_1^{53/40} }}.
		 \end{align}
		\section{Final estimates}
		We have
		\begin{align*}
			S_{r}(N)\,= \mathcal{S}_1 + \mathcal{S}_2.
		\end{align*}
		By using the bound  \eqref{z} and  \eqref{x}, we get
		\begin{align*}
			S_{r}(N)\,\ll \,&\frac{N}{{t^{3/20} p_2^{3/40} p_1^{53/40} }}+ r^{1/2}\,p^{3/4}_1\,N^{3/4}\,K^{3/4}  \\
		&	+ {r^{1/2}\,p^{1/4}_1\,p^{1/2}_2\,N^{3/4}}\, \left(\frac{K^{5/4}}{t^{1/2}} + K^{3/4} + \frac{t^{2/3}}{K^{1/12}} + \frac{t}{K^{1/2}}\right). 
		\end{align*} 
		By  the optimal value of the parameter $K$ i.e. $K = \left({p_2\,t^{2}}/{p_1}\right)^{2/5}$, we get
		
		\begin{align*}
			S_{r}(N)\,\ll \,r^{1/2}\,p^{9/20}_1\,p^{3/10}_2\,t^{3/5}\,N^{3/4} .
		\end{align*}
Now recall Lemma \ref{AFE}
		\begin{align*}
			L \left( \frac{1}{2}+it, \, \pi \times f \right) \ll_{\pi, \, \varepsilon} \mathcal{Q}^{\varepsilon}\sum_{r \leq \mathcal{Q}^{(1+2\varepsilon)/4}}\frac{1}{r}  \sup_{ N\leq \frac{\mathcal{Q}^{1/2+\varepsilon}}{r^2}}\frac{|S_r(N)|}{N^{1/2}} +\mathcal{Q}^{-2023},
		\end{align*} Putting value of $S_{r}(N)$ in above expression, we get
		\begin{align*}
			L \left( \frac{1}{2}+it, \, \pi \times f \right)& \ll_{\pi,  \, \varepsilon}\,\mathcal{Q}^{1/8+\varepsilon}\,p^{9/20}_1\,p^{3/10}_2\,t^{3/5}\,\sum_{r \leq \mathcal{Q}^{(1+2\varepsilon)/4}}\frac{1}{r}\, +\mathcal{Q}^{-2023}\\
			&\ll_{\pi, \, \varepsilon}\,\mathcal{Q}^{1/8+\varepsilon}\,p^{9/20}_1\,p^{3/10}_2\,t^{3/5}. 
		\end{align*} So, finally we get
		\begin{align}\label{s10}
			L \left( \frac{1}{2}+it, \, \pi \times f \right) \ll_{\pi, \, \varepsilon}\,\, \left(\frac{p_1}{p_2\,t^2}\right)^{3/40}\,{\mathcal{Q}^{1/4+\varepsilon}}.
		\end{align} 
		The above bound is subconvex in the $ p_2 <p_1 < p_2t^2$ range. This completes the proof of  Theorem \ref{main theorem}.
		
\vspace{0.4cm}		
{\bf Acknowledgement:} 
The authors are grateful to the anonymous referee for numerous helpful comments, which significantly improved the presentation of this paper.   Mohd Harun would like to thank the Indian Institute of Technology Kanpur for its wonderful academic atmosphere. Sumit Kumar would like to thank HUN-REN Alfr\'ed R\'enyi Institute of Mathematics for providing an excellent environment for completing this work. During this work, S. K. Singh was partially supported by D.S.T. Inspire faculty fellowship no. DST/INSPIRE/04/2018/000945.

		{}

	\end{document}